\documentclass[a4paper, 11pt]{article} 

\usepackage{amsmath, amscd, amsthm, amssymb, mathrsfs} 
\usepackage[applemac]{inputenc}
\usepackage{geometry}
\usepackage[all]{xy}

\DeclareMathOperator{\im}{Im}

\DeclareMathOperator{\ch}{ch}
\DeclareMathOperator{\ind}{ind}

\DeclareMathOperator{\End}{End}

\DeclareMathOperator{\Hom}{Hom}

\DeclareMathOperator{\tr}{Tr}

\DeclareMathOperator{\coker}{coker}

\DeclareMathOperator{\Diff}{Diff}

\DeclareMathOperator{\SO}{SO}
\DeclareMathOperator{\Spin}{Spin}

\DeclareMathOperator{\SU}{SU}

\DeclareMathOperator{\Ham}{Ham}
\DeclareMathOperator{\HVect}{HVect}

\DeclareMathOperator{\Scal}{Scal}

\DeclareMathOperator{\Ad}{Ad}
\DeclareMathOperator{\Lie}{Lie}

\DeclareMathOperator{\Hor}{Hor}

\newcommand{\R}{\mathbb R}
\newcommand{\C}{\mathbb C}

\newcommand{\E}{\mathcal E}
\newcommand{\D}{\mathcal D}
\newcommand{\A}{\mathcal A}
\newcommand{\s}{\mathcal S}
\newcommand{\G}{\mathcal G}
\newcommand{\F}{\mathcal F}
\newcommand{\T}{\mathcal T}
\newcommand{\V}{\mathcal V}

\newcommand{\Id}{\text{Id}}

\newcommand{\diff}{\text{\rm d}}
\newcommand{\del}{\partial}
\newcommand{\delb}{\bar{\del}}

\newcommand{\dvol}{\mathrm{dvol}}

\newcommand{\so}{\mathfrak{so}}

\renewcommand{\P}{\mathbb P}

\renewcommand{\O}{\mathcal O}

\theoremstyle{plain}
	\newtheorem{theorem}{Theorem}
	\newtheorem{proposition}[theorem]{Proposition}
	\newtheorem{lemma}[theorem]{Lemma}
	\newtheorem{corollary}[theorem]{Corollary}
	\newtheorem{conjecture}[theorem]{Conjecture}
	
\theoremstyle{definition}
	\newtheorem{definition}[theorem]{Definition}
	\newtheorem{remark}[theorem]{Remark}

\theoremstyle{plain}
	\newtheorem*{theorem*}{Theorem}
	\newtheorem*{proposition*}{Proposition}
	\newtheorem*{lemma*}{Lemma}
	\newtheorem*{corollary*}{Corollary}
	\newtheorem*{conjecture*}{Conjecture}
\theoremstyle{definition}
	\newtheorem*{definition*}{Definition}
	\newtheorem*{remark*}{Remark}
	\newtheorem*{remarks*}{Remarks}
	
\makeatletter
\def\blfootnote{\xdef\@thefnmark{}\@footnotetext}
\makeatother

\numberwithin{equation}{section}
\numberwithin{theorem}{section}
\begin{document}

\title{A gauge theoretic approach to the anti-self-dual Einstein equations}
\author{Joel Fine}
\date{ }

\maketitle
\vfill
\thispagestyle{empty}
\begin{abstract}
In \cite{Plebanski1977On-the-separati}, Plebanski reformulated the anti-self-dual Einstein equations with non-zero scalar curvature as a first order PDE for a \emph{connection} in an $\SO(3)$-bundle over the four-manifold. The aim of this article is to place this differential equation in a new framework, in which it is both elliptic and a stationary point of a parabolic flow. To do this, we exploit a link with definite connections (introduced in \cite{Fine2009Symplectic-Cala}) to draw an analogy with instantons and the Yang--Mills flow. This picture leads to a natural conjecture, analogous to one made by Donaldson concerning hyperkähler 4-manifolds \cite{Donaldson2006Two-forms-on-fo}. It also provides a moment-map description of the anti-self-dual Einstein equations with non-zero scalar curvature.
\end{abstract}

\vfill
\vfill
\newpage
~
\vfill

\tableofcontents

\vfill
\vfill
\section{Introduction}

In \cite{Plebanski1977On-the-separati}, Plebanski reformulated the anti-self-dual Einstein equations with non-zero scalar curvature as a first order PDE for a \emph{connection} in an $\SO(3)$-bundle over the four-manifold $X$. (This particular interpretation of Plebanski's work appears explicitly in \cite{Capovilla1990Gravitational-i}.) From the physical perspective, there has been renewed interest in this description thanks in part to the recent works of Krasnov \cite{Krasnov2011Gravity-as-a-di,Krasnov2011Pure-Connection}. 

The principal goal of this article is to place this differential equation in a new framework, in which it is both elliptic and a stationary point of a parabolic flow. To do this, we exploit a link with definite connections (which were introduced in \cite{Fine2009Symplectic-Cala}) to draw an analogy with instantons and the Yang--Mills flow. This picture leads to a natural conjecture, analogous to one made by Donaldson concerning hyperkähler 4-manifolds; see Conjectures \ref{skd_conjecture} and \ref{positive conjecture} below. It also provides a moment-map description of the anti-self-dual Einstein equations. 

We establish various foundational results concerning this description of the anti-self-dual Einstein equations. For example we show that the gauge-fixed linearisation of the equation is a Dirac-type operator and compute its index. We also prove local stability of the parabolic flow: if it is started sufficiently close to a solution of the elliptic equation, then the flow exists for all time and converges to the solution, modulo gauge. 

A secondary goal of the article is to popularise this gauge theoretic approach to the anti-self-dual Einstein equations which, despite its elegance, seems not to be so well known in the mathematical community.

\subsection{Acknowledgements} 

I would like to thank Dmitri Panov for numerous important conversations about this article and related topics. I am grateful to Michael Singer for helpful discussions and in particular for explaining Proposition \ref{LC} to me. Conversations with Claude LeBrun have been equally helpful for the twistor theory in  \S\ref{twistor fact}. I would like to thank Frederik Witt and Hartmut Wei{\ss} for interesting discussions about parabolic flows. Finally, I would also like to thank Kirill Krasnov for bringing the article of Plebanski to my attention after a first draft of this article was written, for kindly explaining his own research on this topic to me and for encouraging me in my own work. 

This article was partly completed whilst I was a guest at the Simons Center for Geometry and Physics at Stony Brook University, New York. I would like to thank the staff and other visitors there at that time for providing an ideal environment for carrying out and discussing research.

\subsection{Definite triples and hyperkähler metrics}

The approach in this article originated in part from a conjecture of Donaldson  \cite{Donaldson2006Two-forms-on-fo}. Before stating Donaldson's conjecture, we recall a definition from 4-dimensional linear algebra which will play a central role in all that follows. The wedge product defines a symmetric bilinear form on $\Lambda^2 \R^4$ with values in $\Lambda^4 \R^4$. The form has type $(3,3)$ (although a choice of orientation is required to distinguish the positive directions).  

\begin{definition}
A 3-dimensional subspace $V \subset \Lambda^2 \R^4$ is called \emph{definite} if the wedge product restricts to a definite form on $V$. 
\end{definition}

\begin{conjecture}[Donaldson \cite{Donaldson2006Two-forms-on-fo}]\label{skd_conjecture}
Let $X$ be a compact 4-manifold which admits a triple of symplectic forms $\omega_1, \omega_2, \omega_3$. Suppose that this triple spans a definite 3-plane in $\Lambda^2$ at each point of $X$. Then $X$ admits a hyperkähler metric. (And hence $X$ is diffeomorphic to $T^4$ or a K3 surface.)
\end{conjecture}

Of course, the triple of Kähler forms associated to a hyperkähler metric is an example of such a definite triple. In \cite{Donaldson2006Two-forms-on-fo}, Donaldson suggests that one approach to proving Conjecture \ref{skd_conjecture} is to take the given definite triple $\omega_i$ and attempt to deform it (say, via a continuity method) until it becomes a hyperkähler triple. In \S\ref{flow for definite triples} we describe a geometric flow which attempts to carry out this deformation. We prove short time existence for the flow and that the only possible fixed point is a hyperkähler triple. 

\subsection{Definite connections and anti-self-dual Einstein metrics}

In four dimensions, hyperkähler metrics are the anti-self-dual metrics with zero scalar curvature. The central theme of this article is an analogous version of the set-up considered by Donaldson which applies to anti-self-dual Einstein metrics of \emph{non-zero} scalar curvature. The role of definite triples is played by definite connections. These are connections whose curvature satisfies the following inequality. (For more on definite connections, see~\cite{Fine2009Symplectic-Cala}.)

\begin{definition}\label{definition definite connection}
Let $E \to X$ be an $\SO(3)$-bundle over a 4-manifold. A metric connection $A$ in $E$ is called \emph{definite} if its curvature is non-zero on every tangent 2-plane. I.e., $F_A(u,v) \neq 0$ whenever $u,v$ are linearly independent tangent vectors.
\end{definition}

To spell out the analogy with definite triples of symplectic forms, let $e_1, e_2, e_3$ be a local orthonormal frame for the bundle $\so(E)$ of Lie algebras. (We use the Killing form as the metric on $\so(E)$.) Locally, the curvature of $A$ is given by $F_A=\sum F_i \otimes e_i$, for a triple of 2-forms $F_i$. $A$ is definite if and only if the 2-forms $F_i$ span a definite 3-plane in $\Lambda^2$ at each point. Meanwhile, the closed condition $\diff \omega_i = 0$ has been replaced by the Bianchi identity $\diff_AF_A=0$. 

Next we explain how a definite connection in $E \to X$ determines a Riemannian metric on $X$. 

\begin{definition}\label{definition metric from connection}
Given a definite connection $A$ we define a Riemannian metric $g_A$ on $X$ as follows. We declare the span $\langle F_i\rangle$ to be the bundle of self-dual 2-forms; this defines a conformal class on $X$. (This is the unique conformal class for which $A$ is a self-dual instanton.) We take as volume form $\mu = \frac{1}{3}\sum F_i^2$. It is straightforward to check this definition doesn't depend on the choice of local orthonormal frame of $\so(E)$.

Alternatively, the metric can be described invariantly. Interpreting the curvature $F_A \in \Lambda^2 \otimes \so(E)$ as a homomorphism $\so(E)^* \to \Lambda^2$, $A$ is definite precisely when the image of this map is a definite 3-plane, which we then take to be the bundle $\Lambda^+_A$ (the subscript reminding us of the $A$-dependence). Meanwhile, the volume form is simply the multiple $\frac{8\pi^2}{3}p_1(A)$ of the first Pontrjagin form of  $A$.
\end{definition}

We now explain the analogue of the hyperkähler condition $\omega_i \wedge \omega_j = \delta_{ij} \mu$. This was first considered by Plebanski \cite{Plebanski1977On-the-separati} and subsequently by Capovilla--Jacobson--Dell \cite{Capovilla1991Self-dual-2-for,Capovilla1990Gravitational-i} (although not in the context of definite connections).

\begin{definition}
\label{definition perfect connection}
Given a definite connection $A$, its curvature, considered as a map $F_A \colon \so(E)^* \to \Lambda^2$ identifies $\so(E)^* \cong \Lambda^+_A$. We call $A$ \emph{perfect} if this map is an isometry. Equivalently, if $e_i$ is a local orthonormal framing for $\so(E)$, a definite connection $A$ is perfect if its curvature $F_A= \sum F_i \otimes e_i$ satisfies $F_i \wedge F_j = \delta_{ij} \mu$. 
\end{definition}
 
The following result seems to have been overlooked to some extent by the mathematical community. This may in part be down to the fact that the full proof is spread over three articles (\cite{Plebanski1977On-the-separati,Capovilla1991Self-dual-2-for,Capovilla1990Gravitational-i}), and is given in  terminology mathematicians may be unfamiliar with. Because of this, and in an attempt to popularise this work, we give a self-contained proof in \S\ref{perfect connections}. (This also serves to fix notation and ideas for the later parts of the paper.)

\begin{theorem*}[Plebanski, Capovilla--Jacobson--Dell]
Let $A$ be a perfect connection over a 4-manifold. Then the corresponding metric $g_A$ is anti-self-dual and Einstein with non-zero scalar curvature.
\end{theorem*}

We make two remarks about this Theorem. Firstly, the condition that a connection $A$ be perfect is a first order PDE for $A$. On the other hand, given an arbitrary Riemannian metric $g$, the requirement that $g$ be anti-self-dual and Einstein is second order in $g$. Given a definite connection, the Riemannian metric $g_A$ is also first order in $A$, so we have replaced an a~priori third order equation with a first order one. 

Secondly, as equations for a metric, both the condition that $g$ be Einstein and the condition that $g$ be anti-self-dual are elliptic (modulo diffeomorphisms). Taken together then, the requirement that $g$ be anti-self-dual and Einstein is over-determined. Meanwhile, as we prove in \S\ref{elliptic}, the requirement that a definite connection be perfect is, modulo gauge, an elliptic equation. 

It is interesting to note that when $A$ is perfect, the relevant linear elliptic operator is a Dirac operator associated to the metric $g_A$. We use this observation to compute the index of the operator, via the Atiyah--Singer index theorem; see \S\S\ref{dirac operator} and~\ref{index}. 

\subsection{A conjecture for positive definite connections}

We can now ask the analogue of Donaldson's conjecture: does the existence of a definite connection imply the existence of an anti-self-dual Einstein metric? In fact, this is two seemingly quite different questions, depending crucially on a sign. 

In contrast to arbitrary connections, it is possible to give a sign to the curvature of a definite connection. To define this sign we first recall that $\so(3)$ is naturally oriented. If $e_1, e_2$ are linearly independent, then we declare $e_1, e_2, [e_1,e_2]$ to be an oriented basis. Equivalently, and more invariantly, given a metric and orientation on $\R^3$, the cross-product defines an identification $\R^3 \cong \so(3)$ and hence induces an orientation on $\so(3)$. Reversing the orientation on $\R^3$ changes this identification by a sign and so leaves unchanged the induced orientation on $\so(3)$.

A consequence of this is that on an oriented Riemannian 4-manifold, the bundles $\Lambda^{\pm}$ are naturally oriented. This is because under the metric identification $\Lambda^2 \cong \so(4)$, the splitting $\Lambda^2 = \Lambda^+ \oplus \Lambda^-$ corresponds to the Lie algebra isomorphism $\so(4) \cong \so(3) \oplus \so(3)$. With this observation in hand we can now define the sign of a definite connection.

\begin{definition}
Given a definite connection $A$, its curvature gives an isomorphism $\so(E)^* \to \Lambda^+_A$ between two oriented bundles. We say that the connection is positively or negatively curved according to whether this isomorphism is orientation preserving or reversing respectively. 

Equivalently, if $e_1, e_2, e_3$ is a local oriented frame for $\so(E)$ and $F_A = \sum F_i \otimes e_i$ then $A$ is called positively curved if $F_1, F_2, F_3$ is an oriented frame for $\Lambda^+$ and negatively curved otherwise.
\end{definition}

When a definite connection is perfect, and hence determines an anti-self-dual Einstein metric, this sign agrees with that of the metric's scalar curvature. 

The negatively curved version of Conjecture~\ref{skd_conjecture} is \emph{false}: there are examples of compact 4-manifolds which admit negative definite connections but which do not admit an anti-self-dual Einstein metric. The examples come from a collection of Riemannian 4-manifolds constructed by Gromov and Thurston \cite{Gromov1987Pinching-consta}. Put briefly, they begin with a hyperbolic 4-manifold $M$ containing a null-homologous totally geodesic surface $\Sigma$. Next they consider  the $k$-fold cover $M_k \to M$ ramified along $\Sigma$. On the one hand, pulling back the hyperbolic metric from $M$ and smoothing along the branch locus, they produce a metric on $M_k$ of negative sectional curvature. On the other hand, finiteness considerations show that at most finitely many of the $M_k$ can admit hyperbolic metrics. (Indeed, Gromov and Thurston speculate that none of the $M_k$ are hyperbolic.)

Now, it is shown in \S3.4 of \cite{Fine2009Symplectic-Cala} that for these metrics constructed by Gromov and Thurston, the Levi-Civita connection on either $\Lambda^+$ or $\Lambda^-$ is negative definite. Meanwhile, if $M_k$ admits an anti-self-dual Einstein metric it must be hyperbolic. To see this note that, since $\Sigma$ is null-homologous, the signature of $M_k$ is zero (see, e.g., equation (15) in the article \cite{Hirzebruch1969The-signature-o} of Hirzebruch). For a general 4-manifold, Chern--Weil theory equates the signature to the difference of the $L^2$-norms of the self-dual and anti-self-dual Weyl curvatures. From this it follows that any anti-self-dual metric on $M_k$ is necessarily conformally flat. In particular, if a metric is anti-self-dual and Einstein it must be hyperbolic. 

Whilst the negatively curved version of Conjecture~\ref{skd_conjecture} is false, there are (admittedly speculative) reasons to hope that the positively curved version of the Conjecture may be true. We optimistically state it here:

\begin{conjecture}\label{positive conjecture}
Let $X$ be a compact 4-manifold which admits a positive definite connection. Then $X$ also admits an anti-self-dual Einstein metric with positive scalar curvature. In particular (by a theorem of Hitchin \cite{Hitchin1981Kahlerian-twist}) $X$ is diffeomorphic to $S^4$ or $\C\P^2$. 
\end{conjecture}

The line of reasoning which led to this began in \cite{Fine2009Symplectic-Cala}. The conjecture, if true, would amount to a type of ``sphere theorem''. Sphere theorems in Riemannian geometry say that only the sphere supports a Riemannian metric whose curvature satisfies certain inequalities. From this point of view, the conjecture above could be described as a ``gauge theoretic sphere theorem''.

\subsection{Symplectic geometry of definite connections}\label{symplectic}

We briefly discuss some speculative motivation for Conjecture \ref{positive conjecture} coming from symplectic geometry. We begin by recalling the link between definite connections and symplectic forms which is explained in detail in \cite{Fine2009Symplectic-Cala}. 

Let $E \to X$ be an $\SO(3)$-bundle over a 4-manifold and let $A$ be a metric connection. We denote by $Z \to X$ the associated $S^2$-bundle. The vertical tangent bundle to $Z$ is an $\SO(2)$-bundle $V \to Z$. If we pick an orientation for the fibres of $E$, we can think of $V$ as a Hermitian line bundle. It carries a natural unitary connection defined as follows. A section of $V$ is a vector field on $Z$ which is tangent to the fibres of $Z \to X$. Along the fibres, we can differentiate the section using the Levi-Civita connection on $S^2$. Meanwhile we use horizontal transport with respect to $A$ to identify nearby fibres of $Z \to X$; in this way we can differentiate sections of $V$ in horizontal directions. (The  connection in $V$ is described in more detail in Definition \ref{induced unitary connection}.) The result is a unitary connection $A_V$ in $V$ which depends on the initial choice of connection~$A$.  

The curvature $F(A_V)$ of $A_V$ determines a closed real 2-form on $Z$ via $\omega_A = \frac{ i}{2\pi} F(A_V)$. As the notation suggests $\omega_A$ depends on the initial $\SO(3)$-connection $A$. The crucial point for us is the following result (proved in \cite{Fine2009Symplectic-Cala}).

\begin{proposition}
The closed 2-form $\omega_A$ is symplectic if and only if $A$ is definite. When $A$ is positive definite, $(Z, \omega_A)$ is ``Fano'' in the sense that $c_1(Z) = 2[\omega_A]$. When $A$ is negative definite, $(Z, \omega_A)$ is ``Calabi--Yau'' in the sense that $c_1(Z) = 0$.
\end{proposition}

Certainly in algebraic geometry Calabi--Yau varieties and Fano varieties behave very differently. For example there are precisely 105 different deformation families of smooth algebraic Fano threefolds \cite{Iskovskikh1977Fano-3-folds-I,Iskovskikh1978Fano-3-folds-II,Mori1981Classification-,Mori2003Erratum-Classif}. Smooth algebraic Calabi--Yau threefolds, on the other hand, are far more plentiful. Thousands of topologically distinct examples are currently known \cite{Reid2002Update-on-3-fol}, although the question of whether the total number is finite is still open.

One might also expect such differences between symplectic ``Fanos'' and ``Calabi--Yaus''. Symplectic Calabi--Yaus in real dimension six are known to be much more abundant than those arising in algebraic geometry \cite{Fine2011The-diversity-o}. On the other hand, there is no known example of a symplectic non-algebraic Fano in real dimension 6. (To date, the only known examples of symplectic non-algebraic Fanos start in real dimension 12 \cite{Fine2010Hyperbolic-geom}.) 

With this in mind, it seems plausible to conjecture that a symplectic Fano $Z$ arising via a positive definite connection is in fact algebraic. Given the topology of $Z$ as an $S^2$-bundle over a 4-manifold and certain other numerical facts (most notably that the canonical bundle has a square root) it is not difficult to check that the only possibilities on the list of smooth algebraic Fano threefolds are $\C\P^3$ or the complete flag $F(\C^3)$. (See \S6 of \cite{Fine2009Symplectic-Cala} for details.) Meanwhile, given a perfect positive-definite connection, a result of Hitchin \cite{Hitchin1981Kahlerian-twist} guarantees the corresponding metric is the standard Einstein metric on either $S^4$ or $\C\P^2$. From here it is straightforward to check that the resulting symplectic manifold is $\C\P^3$ or $F(\C^3)$ respectively. 

Suppose then that one could prove Conjecture \ref{positive conjecture} by finding a path of positive-definite connections joining the given one to a perfect connection. This would give a family of cohomologous symplectic forms on $Z$ deforming the symplectic structure to the standard one on $\C\P^3$ or $F(\C^3)$. By Moser's theorem, all the symplectic forms would be equivalent up to diffeomorphism, in particular confirming the conjecture that $Z$ is algebraic. In the other direction, if one could prove that $Z$ were algebraic it would provide strong evidence for Conjecture \ref{positive conjecture}.

\subsection{Energy and a flow for definite connections}

We turn now to an energy functional whose topological minimum, when it exists, corresponds to an anti-self-dual Einstein metric of non-zero scalar curvature. As we will see, the functional is similar in some sense to the Yang--Mills functional. Let $E \to X$ be an $\SO(3)$-bundle. A choice of fibrewise orientation gives an isometry $E \cong \so(E)$. In all that follows we freely identify $E \cong \so(E) \cong \so(E)^* \cong E^*$. For example, the curvature of a connection $F_A \in \Lambda^2 \otimes \so(E)$ will frequently be interpreted as a homomorphism $F_A \colon E \to \Lambda^2$. When the connection is definite, and so determines a Riemannian metric $g_A$, the curvature gives an isomorphism $E \to \Lambda^+_A$. Similarly, we write $S^2E$ for both the bundle of symmetric bilinear forms on $E^* \cong E$ and the bundle of self-adjoint endomorphisms of $E$; we freely identify sections of the two using the fibrewise metric on~$E$.

\begin{definition}\label{Q}
Let $A$ be a definite connection. Pulling back the Riemannian inner-product from $\Lambda_A^+$ determines a symmetric bilinear form on $E$ which we denote by $Q(A)$. Explicitly, for $u,v \in E$,
\[
Q(A)(u,v) = \frac{F_A(u)\wedge F_A(v)}{\mu(A)}.
\] 
(Recall that $\mu(A) = \frac{8\pi^2}{3}p_1(A)$ is the volume-form of $g_A$.) If $e_1, e_2, e_3$ is a local orthonormal frame for $E$ and $F_A = \sum F_i \otimes e_i$ then with respect to this basis $Q(A)$ corresponds to the matrix
\[
Q(A)_{ij} = \frac{F_i \wedge F_j}{\mu(A)}.
\]
\end{definition}

Note that $A$ is perfect precisely when $Q(A)$ agrees with the original fibrewise inner-product on the $\SO(3)$-bundle $E$. Equivalently, $A$ is perfect when $Q(A)_{ij} = \delta_{ij}$.

Now, the original inner-product in $E$ enables us to define $|Q(A)|^2$. Locally, in terms of the matrix $Q(A)_{ij}$,  $|Q(A)|^2= \sum_{ij} Q(A)_{ij}Q(A)_{ji}$. 

\begin{definition}\label{definition of energy}
We define the \emph{energy} of a definite connection $A$ to be
\[
\E(A) = \int_X |Q(A)|^2\, \mu(A).
\]
\end{definition}

\begin{remark}
This functional and generalisations have also been considered by Krasnov in a series of interesting works \cite{Krasnov2011Gravity-as-a-di,Krasnov2011Pure-Connection}. Krasnov is interested in a possible quantisation of Einstein's equations, in which the variable is the connection $A$, rather than the metric. In particular, \cite{Krasnov2011Pure-Connection} shows that Einstein metrics (not necessarily anti-self-dual as in this work) can be described as coming from connections that are critical points of the functional
\[
\E'(A) = \int_X \left(\tr \sqrt{Q(A)}\right)^2 \mu(A)
\]
This gives an interesting reinterpretation of the Einstein equations in terms of connections which is worthy of further study.
\end{remark}

\begin{proposition}\label{topological lower bound}
There is a topological lower bound $\E(A) \geq 8 \pi^2 p_1(E)$. This lower bound is realised if and only if $A$ is perfect.
\end{proposition}
\begin{proof}
Let $e_1, e_2, e_3$ be a local orthonormal frame for $E$. By definition,
\[
\mu(A) = \frac{1}{3}\sum F_A(e_i)\wedge F_A(e_i) = \frac{8\pi^2}{3}p_1(A).
\]
Hence $\tr Q(A) = 3$ and so we may write $Q(A) = \Id +Q_0(A)$ as trace and trace-free parts. Hence
\[
\int_X |Q(A)|^2\, \mu(A)
=
8 \pi^2 \int_X p_1(A) + \int_X |Q_0(A)|^2\, \mu(A).
\]
It follows that $\E(A) \geq 8 \pi^2 p_1(E)$ with equality if and only if $Q_0(A)=0$, i.e., if and only if $Q(A)=\Id$ which is equivalent to $F_A \colon E \to \Lambda^+_A$ being an isometry.
\end{proof}

This result is reminiscent of Yang--Mills theory over 4-manifolds. There, one considers metric connections in an $\SO(3)$-bundle $E \to X$ over a Riemannian manifold. The Yang--Mills functional---the $L^2$-norm of the curvature tensor---is bounded below by $8\pi^2 p_1(E)$ and this bound is realised precisely by instantons. There are two immediate distinctions between that situation and the one considered here. Firstly, for definite connections, the Riemannian metric is no longer fixed, rather it depends on the connection. (Note that, by definition, $A \in \D$ is a $g_A$-instanton, indeed $g_A$ is the unique metric which makes $F_A$ self-dual.) Secondly, in Yang--Mills theory, the bundle $E$ is auxiliary and can be chosen arbitrarily. In our situation this is no longer the case. The existence of a definite connection means that $E$ is isomorphic to $\Lambda^+_A$. Since the space of conformal classes is connected, the corresponding bundles of self-dual 2-forms are always isomorphic. This means that the isomorphism class of $E$ is predetermined.

Notice that this implies $p_1(E) = p_1(\Lambda^+)= 2\chi(X) + 3 \tau(X)$ where $\chi(X)$ and  $\tau(X)$ are the Euler characteristic and signature of $X$ respectively. Since, for a definite connection, the volume form is given by $\mu(A) = \frac{8\pi^2}{3}p_1(A)$ we deduce the following result (which first appeared in \cite{Fine2009Symplectic-Cala}).

\begin{lemma}\label{half of Hitchin-Thorpe}
Let $X$ be a compact 4-manifold which admits a definite connection. Then $2\chi(X) + 3\tau(X) >0$.
\end{lemma}

Note that the condition $2\chi(X)+ 3\tau(X) >0$ is necessary for $X$ to admit an Einstein metric with non-zero scalar curvature. Indeed, the Hitchin--Thorpe inequality states that when $X$ admits an Einstein metric with non-zero scalar curvature then the stronger inequality $2\chi(X) > 3|\tau(X)|$ holds \cite{Hitchin1974Compact-four-di,Thorpe1969Some-remarks-on}.

We next consider the gradient flow of $\E \colon \D \to \R$. For this we need to define a Riemannian metric on $\D$. Since $\D$ is defined by a strict inequality, it is an open set in the space of all metric connections in $E$ (for, say, the $C^\infty$-topology). So the tangent space at $A \in \D$ is $T_A\D = \Omega^1(X, \so(E))$. Since $A$ also defines a Riemannian metric $g_A$ on $X$ we can use the $L^2(g_A)$-inner-product on $T_A\D$ to define a Riemannian metric on $\D$. For $a, b \in T_A\D$, we define their inner-product to be
\[
\langle a,b \rangle
=
\int_X
(a,b)_{g_A}\, \mu(A),
\]
where $(a,b)_{g_A}$ is the pointwise inner-product on $\Lambda^1 \otimes \so(E)$ determined by~$g_A$.

With this definition in hand, the downward gradient flow of $\E$ makes sense. As we explain in Proposition \ref{equation for gradient flow}, the flow is a sort of twisted version of the Yang--Mills flow. This reinforces the idea that $\E$ is, in some sense, analogous to the Yang--Mills functional. In any case, adapting ideas from the Yang--Mills flow, we prove the following result. (The proof hinges on showing that the flow is parabolic modulo gauge.)  

\begin{theorem*}
~
\begin{enumerate}
\item
Given any definite connection $A_0$, the downward gradient flow of $\E$ starting at $A_0$ exists for short time.
\item
The flow is unique for as long as it exists.
\item
If $A$ is a perfect connection and $A_0$ is sufficiently close to $A$ then the flow starting at $A_0$ exists for all time and converges modulo gauge to $A$. 
\end{enumerate}
\end{theorem*}

For more precise statements and the proofs see \S\ref{exists_unique} (existence and uniqueness) and \S\ref{local_stability_section} (local stability).

There is a near-identical discussion for the definite triples of symplectic forms which appear in Donaldson's Conjecture \ref{skd_conjecture}. We describe the corresponding energy functional and its gradient flow in \S\ref{flow for definite triples}.

An immediate and important question to be addressed when using variational methods to attempt to solve a PDE is to decide if the energy functional has any potential critical points besides the sought-after topological minimum. In the case of definite triples we are able to rule out such intermediate critical points; Proposition \ref{unique critical point} gives that, for a compact 4-manifold, the only possible critical point is a hyperkähler triple. So far we do not know if an analogous result holds in the case of definite connections.

\subsection{A moment-map interpretation}

We finish the article in \S\ref{moment map interpretation} with a discussion of another symplectic aspect of definite connections, this time in an infinite dimensional setting. As we will explain, there is an infinite dimensional symplectic manifold $\s$ in which the space of definite connections embeds $\D \subset \s$ as an isotropic subspace. There is an infinite-dimensional group which acts by symplectomorphisms on $\s$ with a moment map $m$. The perfect connections, which determine anti-self-dual Einstein metrics, are precisely the points of $m^{-1}(0) \cap \D$. It will be interesting to see if the moment-map perspective can provide insight into this problem, much as it has done for other geometric PDEs.

\section{Perfect connections}\label{perfect connections}

In this section we prove the following result, originally due to Plebanski and then reinterpreted by Capovilla--Jacobson--Dell. The original proof uses terminology which is perhaps unfamiliar to mathematicians; moreover the calculations are spread over three papers \cite{Plebanski1977On-the-separati,Capovilla1990Gravitational-i,Capovilla1991Self-dual-2-for}. For this reason we give a self-contained proof here, using the notation defined above. Our proof also serves to introduce ideas central to the later study of the equation over the space of definite connections as well as  highlighting the analogy with hyperkähler metrics. 

\begin{theorem}[Plebanski, Capovilla--Jacobson--Dell]
\label{perfect implies asdE}
Let $A$ be a perfect connection over a 4-manifold. Then the corresponding metric $g_A$ is anti-self-dual and Einstein with non-zero scalar curvature.
\end{theorem}

(See definitions \ref{definition definite connection}, \ref{definition metric from connection} and \ref{definition perfect connection} for explanations of definite connections, the metric $g_A$ associated to a definite connection $A$, and perfect connections.)

\subsection{The Levi-Civita connection for a perfect connection}

We suppose that $A$ is a perfect connection. The first step in the proof of Theorem \ref{perfect implies asdE} is to prove the following Lemma, identifying the Levi-Civita connection of $g_A$.

\begin{lemma}\label{perfect is LC}
Let $A$ be a perfect connection in an $\SO(3)$-bundle $E$. If we use $F_A$ to identify $E \cong \Lambda^+_A$, then $A$ is identified with the Levi-Civita connection on $\Lambda^+_A$.
\end{lemma}

The key is the following standard result from 4-dimensional Riemannian geometry, which was explained to us by Michael Singer. We give the proof for lack of an explicit reference. Let $X$ be an oriented Riemannian 4-manifold. Recall that a connection on $\Lambda^+$ is called \emph{torsion-free} if the composition of covariant differentiation $C^\infty(\Lambda^+) \to C^\infty(T^*X\otimes \Lambda^+)$ with skew-symmetrisation $s \colon T^*X \otimes \Lambda^+ \to \Lambda^3$ equals the exterior derivative. For an arbitrary connection on $\Lambda^+$, the difference $s \circ \nabla - \diff$ is a section of $\Lambda^3 \otimes (\Lambda^+)^*$, called the torsion of $\nabla$, which we write as $\tau$.

\begin{proposition}\label{LC}
Let $X$ be an oriented 4-manifold with Riemannian metric $g$. The Levi-Civita connection on $\Lambda^+$ is the unique torsion-free metric connection.
\end{proposition}

\begin{proof}
By definition, the Levi-Civita on $\Lambda^+$ is metric and torsion-free, hence it suffices to show uniqueness. Suppose that $\nabla \colon \Gamma(\Lambda^+) \to \Gamma(\Lambda^+\otimes\Lambda^1)$ is a metric torsion-free connection. Let $\theta_1, \theta_2, \theta_3$ be a local orthonormal basis for $\Lambda^+$. Since $\nabla$ is metric, its connection matrix is given by 1-forms $\alpha_1, \alpha_2, \alpha_3$ satisfying
\begin{eqnarray*}
\nabla\theta_1
	&=&
		\alpha_3\otimes\theta_2- \alpha_2\otimes\theta_3,\\
\nabla \theta_2
	&=&
		- \alpha_3\otimes\theta_1 +\alpha_1\otimes\theta_3,\\
\nabla \theta_3
	&=&
		 \alpha_2\otimes\theta_1	-\alpha_1\otimes\theta_2.
\end{eqnarray*}

The maps $J_i \colon \Lambda^1 \to \Lambda^1$, defined by $J_i(\alpha) = \ast(  \alpha \wedge\theta_i )$ give three (locally defined) almost complex structures on $X$ which satisfy the quaternionic relations. Using this notation and the fact that $\nabla$ is torsion-free gives
\begin{eqnarray*}
~\ast\diff\theta_1
	&=&
		J_2(\alpha_3)-J_3(\alpha_2),\\
~\ast \diff \theta_2
	&=&
		-J_1(\alpha_3)+J_3(\alpha_1),\\
~\ast \diff \theta_3
	&=&
		J_1(\alpha_2)-J_2(\alpha_1).
\end{eqnarray*}
These equations and the quaternionic relations can now be used to determine the $\alpha_i$. For example, $
\alpha_1
=
\frac{1}{2}
\left(
J_2(\ast \diff \theta_3 )
-
J_3 (\ast \diff \theta_2)
-
\ast \diff \theta_1
\right)
$.
It follows that the $\alpha_i$, and hence $\nabla$, are determined entirely by the local orthonormal frame $\theta_i$.
\end{proof}

Distinct conformal classes give rise to distinct bundles $\Lambda^+ \subset \Lambda^2$. Whilst they are all abstractly isomorphic as rank-3 bundles, they are genuinely different when considered as sub-bundles of $\Lambda^2$. We now give a way to simultaneously describe all choices of $\Lambda^+ \subset \Lambda^2$ \emph{and} all torsion-free connections on the choice of $\Lambda^+$. 

Let $V \to X$ be a rank-3 real vector-bundle over a 4-manifold $X$ and $\phi \in \Omega^2(V^*)$ a 2-form with values in $V^*$. Suppose that, when thought of as a bundle homomorphism $\phi \colon V \to \Lambda^2$, the image of $\phi$ is a definite rank-3 bundle. We define a conformal class on $X$ by setting $\Lambda^+$ to be the image of $\phi$. Of course, the existence of such a $\phi$ determines the isomorphism class of $V$. By varying $\phi$ we can realise all possible conformal classes. Now, given a connection $\nabla$ in $V$ we can push it forward via $\phi$ to obtain a connection $\nabla^\phi$ in $\Lambda^+$. All connections in $\Lambda^+$ arise in this way. 

We next give a result which describes those pairs $(\phi, \nabla)$ for which $\nabla^\phi$ is torsion-free on $\Lambda^+ = \im \phi$. The notation $\diff_\nabla \phi \in \Omega^3(V^*)$ means the covariant exterior derivative of $\phi$ with respect to $\nabla$.

\begin{lemma}\label{torsion}
The connection $\nabla^\phi$ is torsion-free if and only if $\diff_{\nabla} \phi = 0$.
\end{lemma}

\begin{proof}
This just involves unwinding the definitions. First, some notation. Given $a \in \Omega^p(V)$ and $b \in \Omega^q(V^*)$, write $a \wedge b \in \Omega^{p+q}$ for the form produced by tensoring the wedge product on forms with the contraction $V^* \otimes V \to \R$. In this notation, the Leibniz law holds in the form
\[
\diff (a\wedge b) = (\diff_\nabla a) \wedge b \pm a \wedge (\diff_\nabla b)
\]
where the sign is determined by the degree of $a$.

Now let $\alpha\in \Lambda^+$ and write $\alpha = \phi(v)$ for $v \in V$. The push-forward $\nabla^\phi$ of $\nabla$ to $\Lambda^+$ is defined by $\nabla^\phi \alpha = \phi(\nabla v)$, whose skew-symmetrisation is $\phi  \wedge \nabla v$. On the other hand, 
\[
\diff \alpha 
\ =\ \diff (\phi(v))
\ =\ \diff ( \phi \wedge v)
\ =\ \diff_\nabla \phi \wedge v + \phi \wedge \nabla v.
\]
It follows that the torsion $\tau$ is given by $\tau(\alpha)=\diff_\nabla \phi \wedge v$. This vanishes for all $\alpha$ if and only if $\diff_\nabla \phi = 0$.
\end{proof}

We now return to perfect connections and give a proof of the fact that
when $A$ is a perfect connection in $E$, the identification $E \to \Lambda_A^+$ provided by $F_A$ matches up $A$ with the Levi-Civita connection.

\begin{proof}[Proof of Lemma \ref{perfect is LC}.]
A definite connection $A$ in an $\SO(3)$-bundle $E$  determines a metric $g_A$ on $X$. Identifying $E \cong \so(E)^*$, the curvature of $A$ gives an isomorphism $F_A \colon E \to \Lambda^+_A$, playing the role of $\phi$ above. Now Lemma \ref{torsion} and the Bianchi identity, $\diff_A F_A =0$, tell us that the push-forward of $A$ to a connection on $\Lambda^+$ is torsion-free. 

If, in addition, $A$ is perfect, then $F_A \colon E \to \Lambda^+_A$ is an isometry. Since $A$ preserves the metric on $E$ it follows that its push-forward to $\Lambda^+$ is both metric preserving and torsion free. By Proposition \ref{LC} this completely characterises the Levi-Civita connection.
\end{proof}

\subsection{Brief recap of anti-self-dual Einstein metrics}

We now pause to recall the definition of anti-self-dual metrics on 4-manifolds. (See, for example, \cite{Atiyah1978Self-duality-in} for details.) Given an oriented Riemannian 4-manifold, the Levi-Civita connection on $\Lambda^+$ has curvature $R \in \Lambda^2 \otimes \so(\Lambda^+)$. Using the metric and orientation to identify $\so(\Lambda^+) \cong (\Lambda^+)^*$ we think of $R$ as a homomorphism $R \colon \Lambda^+ \to \Lambda^2$. The splitting $\Lambda^2 = \Lambda^+ \oplus \Lambda^-$ means that $R$ splits as $R = G\oplus H$ where $G \in \End(\Lambda^+)$ and $H \in \Hom (\Lambda^+, \Lambda^-)$. It turns out that the condition $H = 0$ is equivalent to the metric being Einstein. Meanwhile, $\tr G =\frac{1}{4}\Scal$ is determined by the scalar curvature. When the trace-free part of $G$ vanishes, so that $G = \frac{1}{12}\Scal\cdot\,\text{Id}$, the metric is called \emph{anti-self-dual}. (Doing the same for the Levi-Civita connection on $\Lambda^-$ gives the definition of self-dual metrics.)  Finally, when both $G$ and $H$ vanish, i.e, when $\Lambda^+$ is flat, the metric is (locally) hyperkähler.

\subsection{Hyperkähler metrics via definite triples}

We can now show that Proposition \ref{LC} and Lemma \ref{torsion} suffice to prove the following standard fact, which is the starting point for Donaldson's Conjecture \ref{skd_conjecture}. 

\begin{proposition}[See, for example, \cite{Besse1987Einstein-manifo}]
Let $X$ be a 4-manifold with a triple of symplectic forms $\omega_i$ which satisfy $\omega_i \wedge \omega_j = \delta_{ij}\mu$ for some volume form $\mu$. Define a Riemannian metric $g$ on $X$ by setting $\Lambda^+$ to be the span of the $\omega_i$ and taking $\mu$ as the volume form. Then $g$ is hyperkähler, i.e., the Levi-Civita connection on $\Lambda^+$ is flat, with no monodromy. 
\end{proposition}
\begin{proof}
To prove this, note that the forms $\omega_i$ give an identification $\phi \colon \underline{\R}^3 \to \Lambda^+$ with the trivial bundle. In the above notation, we take $\nabla$ to be the product connection on the trivial bundle $\underline{\R}^3$. The fact that the $\omega_i$ are closed translates into $\diff_\nabla \phi=0$. So, by Lemma \ref{torsion}, pushing forward the product connection gives a torsion-free connection $\nabla^\phi$ on $\Lambda^+$. Now the fact that $\omega_i \wedge \omega_j = \delta_{ij} \mu$ means that under $\phi$ the constant inner-product on $\underline{\R}^3$ is identified with the Riemannian inner-product on $\Lambda^+$. Since the constant inner-product is preserved by $\nabla$ it follows that $\nabla^\phi$ is metric preserving on $\Lambda^+$. Finally, by Proposition \ref{LC}, it follows that $\phi$ matches up with the Levi-Civita connection on $\Lambda^+$ with the product connection on $\underline{\R}^3$ and so $\Lambda^+$ is flat with a global parallel trivialisation.
\end{proof}

\subsection{Anti-self-dual Einstein metrics via definite connections}

We next deduce a similar result from Lemma \ref{perfect is LC}, namely that if $A$ is perfect then $g_A$ is anti-self-dual and Einstein with non-zero scalar curvature. Note that it follows immediately from Lemma \ref{perfect is LC} that $g_A$ is Einstein because the curvature of $E \cong \Lambda^+$ has no $\Lambda^-$-component ($A$ is a self-dual instanton with respect to $g_A$). What remains to be verified is that the metric is also anti-self-dual with non-zero scalar curvature. This will follow from the next result.

\begin{lemma}\label{metric perfect implies asdE}
Let $g$ be an Einstein metric on an oriented 4-manifold such that the Levi-Civita connection on $\Lambda^+$ is perfect. Then $g$ is also anti-self-dual and has non-zero scalar curvature.
\end{lemma}

\begin{proof}
Let $A$ denote the Levi-Civita connection on $\Lambda^+$. We begin by showing that the metric $g_A$ determined by the definite connection $A$ is conformal to the original metric $g$. Recall that the curvature map $F_A \colon \Lambda^+ \to \Lambda^2$ splits as 
\[
F_A =G\oplus H \colon \Lambda^+ \to \Lambda^+ \oplus \Lambda^-
\]
where $G \in \End (\Lambda^+)$ and $H \in \Hom(\Lambda^+, \Lambda^-)$. Now, since $g$ is Einstein, $H=0$. Hence $\im F_A = \Lambda^+(g)$. In other words, the self-dual 2-forms of $g$ and $g_A$ agree and so the two metrics are conformally equivalent.

Next, we show that in fact $g_A = cg$ for some constant $c$. Since $A$ is perfect, for any unit-length $\theta \in \Lambda^+$ the 4-form $G(\theta)^2 = F_A(\theta)^2$ is the volume form $\mu(A)$ of $g_A$ and so does not depend on $\theta$. But if $\theta$ is an unit eigenvector of $G$ with eigenvalue $\lambda$ then $G(\theta)^2 = \lambda^2\dvol(g)$, hence all eigenvalues of $G$ have equal square. The sum of the eigenvalues of $G$ is equal to $\frac{1}{4}\Scal(g)$ which is constant, since $g$ is Einstein. From here it follows that all the eigenvalues of $G$ are constant. This in turn implies that $\dvol(g)/\mu(A)$ is constant and hence that $g_A = cg$ for some constant $c$.

By rescaling $g$ (which does not change $A$ and hence $g_A$) we can assume $g_A= g$. With this choice of scale, $G \colon\Lambda^+ \to \Lambda^+$ is an isometry, hence all eigenvalues $\lambda_i$ satisfy $\lambda_i=\pm 1$. It remains to show that all $\lambda_i$ have the same sign.

The idea is to show that if one of the eigenvalues has a different sign, the corresponding eigenvector $\theta$ is parallel. This gives a contradiction because definite connections never admit parallel sections: such a thing would imply that the curvature of $A$ actually took values in $\Lambda^2 \otimes \so(2)$, and such a connection can never be definite.

To proceed, we let $\theta_i$ be a local frame of $\Lambda^+$ consisting of eigenvectors of~$G$:
\[
F_A = G = \sum \lambda_i\,\theta_i \otimes \theta_i
\]
The Bianchi identity says $\diff_A F_A = 0$, i.e., 
\[
\sum \lambda_i \left(
\diff \theta_i \otimes \theta_i
+
\theta_i \wedge \nabla_A \theta_i \right)
=
0
\]
since the $\lambda_i$ are all constant.

Write the Levi-Civita connection on $\Lambda^+$ in terms of the basis $\theta_i$ as
\begin{eqnarray*}
\nabla_A \theta_1 
	&=&
	\alpha_3\otimes\theta_2-\alpha_2\otimes\theta_3,\\
\nabla_A \theta_2
	&=&
	-\alpha_3\otimes\theta_1+\alpha_1\otimes\theta_3,\\
\nabla_A\theta_3
	&=&
	\alpha_2\otimes\theta_1-\alpha_1\otimes\theta_2.
\end{eqnarray*}
for 1-forms $\alpha_i$. Since $A$  is torsion free, $\diff \theta_i$ is given by the formula for $\nabla_A \theta_i$ with $\otimes$ replaced by $\wedge$. Using this gives 
\[
\sum \lambda_i
(\diff\theta_i\otimes\theta_i + \theta_i \wedge \nabla_A \theta_i) 
= 
A_1 \otimes \theta_1 
+ 
A_2 \otimes \theta_2 
+ 
A_3 \otimes \theta_3
\]
where
\begin{eqnarray*}
A_1 &=&
(\lambda_1 - \lambda_2)(\theta_2 \wedge \alpha_3) 
+ 
(\lambda_3-\lambda_1)(\theta_3 \wedge \alpha_2),\\
A_2 &=&
(\lambda_2-\lambda_3)(\theta_3\wedge\alpha_1)
+
(\lambda_1-\lambda_2)(\theta_1\wedge\alpha_3),\\
A_3 &=&
(\lambda_3 - \lambda_1)(\theta_1\wedge\alpha_2)
+
(\lambda_2 - \lambda_3)(\theta_2 \wedge\alpha_1).
\end{eqnarray*}
The Bianchi identity implies that all the $A_i$ vanish. 

Suppose now, for a contradiction, that only two of the eigenvalues have the same sign; say $\lambda_1=\lambda_2$, but  $\lambda_1 \neq \lambda_3$. Then $A_1 = 0$ implies that $\theta_3 \wedge \alpha_2=0$ and so $\alpha_2 = 0$ (as wedging with $\theta_i$ is an isomorphism $\Lambda^1 \to \Lambda^3$). Similarly, $A_2=0$ implies that $\alpha_1=0$. This means that $\nabla_A \theta_3 = 0$ thus giving a contradiction. 
\end{proof}

We can now prove, as promised in the introduction, that perfect connections yield anti-self-dual Einstein metrics with non-zero scalar curvature.

\begin{proof}[Proof of Theorem \ref{perfect implies asdE}]
Given a perfect connection $A$ in an $\SO(3)$-bundle $E \to X$, the curvature $F_A$ identifies $E \cong \Lambda^+_A$. It follows from Proposition \ref{LC} and Lemma \ref{torsion} that under this identification, $A$ matches up with the Levi-Civita connection $\nabla$ of $g_A$ on $\Lambda^+_A$. Since the curvature of $\nabla$ has no $\Lambda^-$ component, $g$ is Einstein. The result now follows from Lemma \ref{metric perfect implies asdE}.
\end{proof}

\section{Ellipticity and definite connections}\label{elliptic}

The goal in this section is to show that perfect connections are the zeros of a non-linear differential operator defined on the space of definite connections which is \emph{elliptic modulo gauge}.

Recall from Definition \ref{Q} that to each definite connection $A$ we associate a symmetric bilinear form $Q(A)$ on $E$, given by pulling back the Riemannian inner-product on $\Lambda^+_A$ via the isomorphism $F_A \colon E \to \Lambda^+_A$. Explicitly,
\[
Q(A)(u,v) = \frac{F_A(u) \wedge F_A(v)}{\mu_A} =(F_A(u), F_A(v)), 
\]
where $\mu_A = \frac{8\pi^2}{3}p_1(A)$ is the volume form of $g_A$.
So $A$ is perfect precisely when $Q(A)$ equals the original fibrewise metric in $E$. We will show that
\[
Q \colon \D \to \Omega^0(X,S^2E)
\]
is elliptic modulo gauge.

Linearising at the point $A\in \D$ gives a first-order operator
\[
\delta_A Q \colon \Omega^1(X, \so(E)) \to \Omega^0(X,S^2_0E).
\]
(Note $\tr Q(A) = 3$ is constant, so the linearisation takes values in trace-free endomorphisms $S^2_0E$.) This is obviously not elliptic for the simple reason that it is map between sections of bundles of different ranks. This failure of ellipticity can also be seen in terms of gauge transformations. The gauge group $\G$ is the group of all bundle isometries $E \to E$ (not necessarily covering the identity on $X$). It acts on both connections and sections of $S^2_0E$ by pull-back. The map $Q$ is $\G$-equivariant and this prevents $\delta_A Q$ from being elliptic. Put briefly, the infinitesimal action of $\G$ on $\Omega^0(X,S^2_0E)$ is an algebraic map, 
\[
T \colon \Lie(\G) \to \Omega^0(X,S^2_0E),
\]
so of zeroth order when thought of as a differential operator. Writing $R_A \colon \Lie(\G) \to T_A\D$ for the infinitesimal action of $\G$ at $A$, we see that $\delta_A Q \circ R_A = T$. Now $\delta_A Q \circ R_A$ is a~priori at least first order (since $\delta_A Q$ is), so the symbol of $\delta_A Q \circ R_A$ must vanish. In other words, the symbol of $\delta_A Q$ must vanish on the image of the symbol of $R_A$. 

We will show that this is the only way in which $\delta_A Q$ fails to be elliptic. Working orthogonally to the $\G$-orbits in $\D$ corrects for this failure and fixing the gauge in this way yields an elliptic linear operator, as is familiar from the theory of instantons. In the special case when $A$ is perfect, the corresponding elliptic operator can be identified with a certain Dirac operator of the metric $g_A$; see \S\ref{dirac operator}. In \S\ref{index} we exploit this to compute the index of the gauge-fixed operator.

First, however, a reminder about our notation. Throughout, we freely identify $E \cong E^* \cong \so(E) \cong \so(E)^*$. So, for example, the tangent space $T_A\D = \Omega^1(X, \so(E))$ is identified with $\Omega^1(X,E)$.

\subsection{The linearisation of $Q$}

We begin by describing the linearisation of $A \mapsto Q(A)$. To do this we introduce some notation. There is a map 
\[
(\Lambda^2 \otimes E) \otimes (\Lambda^2 \otimes E) \to \Lambda^4\otimes S^2E,
\]
given by tensoring the wedge product on forms with the natural projection $E\otimes E \to S^2E$. We write this map $S \otimes T \mapsto S \cdot T$. In this notation, $Q(A) =(F_A \cdot F_A)/ \mu(A)$.

Using this map, we next define a bundle homomorphism
\begin{equation}\label{L}
L_A \colon \Lambda^2 \otimes E \to S^2E, \quad\quad
L_A(T)
= 
\frac{F_A \cdot T}{\mu(A)}.
\end{equation}
So, in this notation, $Q(A) = L_A(F_A)$. Explicitly, if $e_1, e_2, e_3$ is a local orthonormal frame for $E$ in which $F_A = \sum F_i \otimes e_i$ and $T = \sum T_i \otimes e_i \in \Lambda^2 \otimes E$, then $L_A(T)$ is represented by the symmetric matrix with $(i,j)$-entry 
\[
L_A(T)_{ij} = \frac{T_i\wedge F_j + T_j \wedge F_i}{2\mu(A)}.
\]
Because the forms $F_i$ are, by definition of $g_A$, self-dual, we can write this as
\[
L_A(T)_{ij} = \frac{1}{2}(T_i, F_j) + \frac{1}{2}(T_j, F_i)
\]
where $(\cdot, \cdot)$ is the Riemannian inner-product of $g_A$. Notice that $\tr(L_A(T)) = (F_A, T)$.

With this in hand we have:
\begin{lemma}\label{linearisation of Q}
Let $A \in \D$ and $a \in T_A\D = \Omega^1(X,E)$. The linearisation of $Q$ at $A$ in the direction $a$ is given by:
\begin{equation}\label{delta Q}
\delta_A Q(a) = 2 L_A(\diff_Aa) - \frac{2}{3} (F_A, \diff_A a) Q(A).
\end{equation}
(Here $(\cdot, \cdot)$ is the inner-product on $\Lambda^2 \otimes E$ determined by the given one on $E$ and $g_A$ on $\Lambda^2$). In particular, $\delta_A Q$ takes values in trace-free endomorphisms $S^2_0E$. 
\end{lemma}
\begin{proof}
Differentiating the equation $\mu(A)= \frac{1}{3} \tr (F_A^2)$ along $a$ gives
\begin{equation}\label{delta mu}
\delta_A\mu(a)
=
\frac{2}{3} \tr \left(F_A \wedge \diff_Aa \right)
=
\frac{2}{3} \tr \left[L_A(\diff_Aa)\right]\mu(A).
\end{equation}
Since, as remarked above, $\tr (L_AT)= (F_A, T)$ we can write this as 
\[
\delta_A\mu(a) = \frac{2}{3} (F_A, \diff_A a) \mu(A).
\]
Next, differentiating $Q(A)\mu(A)= F_A \cdot F_A$ gives
\[
\delta_A Q(a)\,\mu(A) + \frac{2}{3} Q(A) (F_A, \diff_A a)\mu(A)
=
2 F_A \cdot \diff_A a  
\]
from which the equation for $\delta_AQ(a)$ follows. 

The formula implies $\tr \delta_AQ(a) = 0$, which it must since $\tr Q =3$ is constant.
\end{proof}

\subsection{The linearised action} \label{linearised action}

Let $\G$ denote the gauge group of all bundle isometries $E \to E$. 
It acts by pull-back on the space of metric connections in $E$, preserving the space of definite connections. Differentiating this action at a connection $A$ gives a map $R_A \colon \Lie(\G) \to \Omega^1(X, E)$. To describe $R_A$, we first consider the subgroup $\G_0$ of gauge transformations covering the identity on $X$ (the usual gauge group in Yang--Mills theory). It has Lie algebra $\Lie (\G_0) = \Omega^0(X, E)$ and here $R_A$ is given by the familiar formula: $R_A(\xi) = -\diff_A \xi$. (Recall we implicitly identify $\so(E) \cong E$ throughout.)

Next, we use the connection $A$ to determine a vector-space complement to $\Lie(\G_0) \subset \Lie (\G)$ by horizontally lifting vector fields on $X$ to $E$. This gives
\[
\Lie(\G) = \Lie(\G_0) \oplus \Hor_A \cong \Lie(\G_0) \oplus C^\infty(TX)
\]
where $\Hor_A \cong C^\infty(TX)$ are the horizontal lifts to $E$ of vector fields on $X$. Of course, $\Hor_A$ is not a Lie subalgebra precisely because $A$ has curvature. 

\begin{lemma}\label{linear action of vector}
Given $u \in C^\infty(TX)$ the infinitesimal action at $A$ of its $A$-horizontal lift is $R_A(u) = - \iota_u F_A$.
\end{lemma}
\begin{proof}
We switch to the principal bundle formalism. Let $P \to X$ be the principal frame bundle of $E$. A connection $A$ is an $\SO(3)$-equivariant 1-form on $P$ with values in $\so(3)$ whilst $\Lie(\G)$ is the Lie algebra of $\SO(3)$-invariant vector fields on $P$. Given \emph{any} element $u \in \Lie (\G)$, the corresponding infinitesimal action on $A$ is $R_A(u) = - L_u(A) = -\diff(A(u)) - \iota_u \diff A$. 

With respect to the splitting of $\Lie(\G)$ determined by $A$, $\Hor_A$ is precisely those $u$ with $A(u) =0$. For such vectors, $\iota_u[A\wedge A] = 2[A(u), A] = 0$. It follows that $\iota_u \diff A = \iota_u F_A$, since $F_A = \diff A + \frac{1}{2}[A\wedge A]$.
\end{proof}

So, given $A$ we have an isomorphism $\Lie(\G) \cong \Omega^0(X, E) \oplus C^\infty(TX)$ with respect to which the infinitesimal action at $A$ is given by 
\begin{equation}\label{infinitesimal action}
R_A(\xi, u) = -\diff_A \xi - \iota_u F_A.
\end{equation}

Next we consider the action of $\G$ by pull-back on $\Omega^0(X, S^2_0E)$. Write the infinitesimal action at a section $B \in \Omega^0(X, S^2_0E)$ as $T_B \colon \Lie(\G) \to \Omega^0(X, S^2_0E)$. Given a connection $A$, we write elements of the Lie algebra as $(\xi,u)$ as above. Then
\begin{equation}\label{infinitesimal G equivariance}
T_B (\xi,u) = [B, \xi] - \nabla^A_u B.
\end{equation}
The map $Q$ is $\G$ equivariant, so
\[
\delta_A Q \circ R_A = T_{Q(A)}.
\]
In other words, $\delta_A Q(\diff_A \xi) = [\xi, Q(A)]$ and $\delta_A Q(\iota_u F_A) = \nabla^A_u Q(A)$.

\subsection{The gauge-fixed elliptic operator}

We now explain how to take account of gauge to produce a linear elliptic operator out of $\delta_A Q$. Given a definite connection $A$, we denote the symbol of the linearisation $\delta_A Q \colon \Omega^1(X, E) \to \Omega^0(X,S^2_0E)$ at $A$ in the direction $\alpha \in \Lambda^1$ by
\[
\sigma_A(\alpha) \colon \Lambda^1 \otimes E
\to
S^2_0E.
\]
It follows from the description of $\delta_A Q$ given in Lemma \ref{linearisation of Q} that 
\begin{equation}\label{formula for symbol}
\sigma_A (\alpha)(a) 
=
2L_A(\alpha \wedge a) - \frac{2}{3} \left(\tr L_A(\alpha \wedge a)\right)Q(A).
\end{equation}

We first prove that $\sigma_A(\alpha)$ is surjective. We begin with a more explicit description of the operator $L_A$. Let $e_1, e_2, e_3$ denote an orthonormal frame of $E$. Writing $F_A = \sum F_i \otimes e_i$, the 2-forms $F_i$ span $\Lambda^+_A$. Hence a self-dual 2-form with values in $E$ can be written as $T = \sum_{ij} T_{ij} F_j \otimes e_i$. In the following lemma we will identify self-dual 2-forms $T$ with 3-by-3 matrices $(T_{ij})$ in this fashion.

\begin{lemma}\label{matrix form of L}
\strut
\begin{itemize}
\item
Given a definite connection $A$, the map $L_A \colon \Lambda^2 \otimes E \to S^2E$ vanishes on $\Lambda^-_A\otimes E$. 
\item
Let $T = \sum T_{ij}F_j \otimes e_i \in \Lambda^+_A \otimes E$. With respect to the basis $e_1, e_2, e_3$, $L_A(T)$ is represented by the matrix $\frac{1}{2}\left(TQ(A) + Q(A)T^t\right)$. In particular, $L_A \colon \Lambda^+_A \otimes E \to S^2E$ is surjective.
\end{itemize}
\end{lemma}
\begin{proof}
Since $L_A$ involves taking the wedge product with self-dual 2-forms it vanishes on elements of $\Lambda^-_A \otimes E$. Next note that by definition, for $T$ as in the statement, $L_A(T)$ corresponds to the matrix with $(p,q)$-element
\[
\frac{1}{2}\sum_j T_{pj}(F_j, F_q)
+
\frac{1}{2} \sum_j T_{qj}(F_p, F_j)
\]
which is the $(p,q)$-element of $\frac{1}{2}\left(TQ(A)+Q(A)T^t\right)$ as claimed. Finally, to prove surjectivity, given a symmetric matrix $M$, let $T = MQ(A)^{-1}$; then $L_A(T)$ is represented by the matrix $M$.
\end{proof}

\begin{corollary}\label{symbol is surjective}
For $\alpha \neq 0$, the symbol $\sigma_A(\alpha)$ is surjective.
\end{corollary}

\begin{proof}
Recall the standard fact that the map $\Lambda^1 \to \Lambda^+$ given by $\beta \mapsto (\alpha \wedge \beta)^+$ is a surjective homomorphism (this is the key to ellipticity in the study of instantons). It now follows from Lemma \ref{matrix form of L}  that the map $\Lambda^1 \otimes E \to S^2E$ given by $a \mapsto L_A((\alpha\wedge a)^+)$ is surjective. Finally, given $M \in S^2_0E$, let $a \in \Lambda^1\otimes E$ satisfy $2L_A(\alpha \otimes a) = M$. Then $\sigma_A(\alpha)(a) = M - \frac{1}{3}(\tr M) Q(A) = M$.
\end{proof}

We next explain how $\sigma_A$ fits into a short exact sequence. Write $S_A \colon TX \to \Lambda^1\otimes E$ for the homomorphism $S_A(u) = \iota_u F_A$, which gives minus the infinitesimal action at $A$ of (the $A$-horizontal lift of) $u$. Meanwhile, minus the infinitesimal action at $A$ of an endomorphism $\xi \in \Omega^0(X, \so(E)) \cong \Omega^0(X, E)$ is $\diff_A\xi$. The symbol of this operator in the direction $\alpha$ is wedge product with $\alpha$, which we denote by $w_\alpha \colon E \to \Lambda^1 \otimes E$.

\begin{lemma}\label{exact sequence}
Let $A$ be a definite connection. Then for any non-zero $\alpha \in \Lambda^1$, the following is a short exact sequence:
\begin{equation}\label{symbol exact sequence}  
\xymatrixcolsep{2.7pc}
\xymatrix{
0\ar[r]
&
TX \oplus E
\ar[r]^-{S_A + w_\alpha}
&
\Lambda^1 \otimes E
\ar[r]^-{\sigma_A(\alpha)}
&
S^2_0E
\ar[r]
&
0}
\end{equation}
\end{lemma}

\begin{proof}
We first prove that $\sigma_A(\alpha) \circ (S_A + w_\alpha) = 0$. Recall equation (\ref{infinitesimal G equivariance}) which says that $\G$-equivariance of $Q$ implies that for any $\xi \in \Omega^0(X, E)$,
\[
\delta_AQ (\diff_A \xi)
=
[\xi, Q(A)]
\]
This is zeroth order in $\xi$ and so the symbol of $\delta_A Q\circ \diff _A$ (a~priori second order) must vanish. In other words, $\sigma_A(\alpha) \circ w_\alpha = 0$. (This is also immediate from the formula for $\sigma_A(\alpha)$).

Similarly, for any $u \in \Omega^0(X, TX)$, the $\G$-equivariance of $Q$ implies that
\[
\delta_A Q(S_A(u))
=
\nabla^A_u Q(A).
\]
Again, this is zeroth order in $u$ and so the symbol of $\delta _AQ \circ S_A$ (a~priori first order) must vanish. In other words, $\sigma(\alpha) \circ S_A = 0$. (This can also be proved directly from the formula for $\sigma_A(\alpha)$ with a little more work; see Remark \ref{direct calc}.)

We have already proved in Lemma \ref{symbol is surjective} that $\sigma_A(\alpha)$ is surjective, so it just remains to prove that $S_A + w_\alpha$ is injective. Given a basis $e_1,e_2,e_3$ for $E$, write $F_A = \sum F_i \otimes e_i$. Let $u \in TX$ and $x = \sum x_ie_i \in E$. Then $(S_A+w_\alpha)(u ,x) = \sum (\iota_uF_i +x_i \alpha)\otimes e_i$ vanishes if and only if $\iota_uF_i +x_i \alpha = 0$ for each $i$. In other words, the 1-forms $\iota_u F_i$ are all proportional. On the other hand, since the $F_i$ span a definite 3-plane in $\Lambda^2$, if $u \neq 0$ the 1-forms $\iota_uF_i$ span a 3-plane in $\Lambda^1$. So $(S_A+w_\alpha)(u,x) = 0$ forces $u=0=x$.
\end{proof}

We can now gauge-fix $\delta_A Q$ to produce an elliptic operator. Let $V_A \subset \Lambda^1 \otimes E$ denote the rank 8 sub-bundle which is orthogonal to the image of $S_A \colon TX \to \Lambda^1 \otimes E$. Write $W = S^2_0E \oplus E$. 

\begin{proposition}
Given a definite connection $A$, the operator
\[
D_A = \delta_A Q \oplus \diff_A^* 
\colon 
\Omega^0(X, V_A ) \to \Omega^0(X, W)
\]
is elliptic. (Here $\diff_A^*$ denotes the $L^2$-adjoint of $\diff_A$ defined via the Riemannian metric $g_A$.)
\end{proposition}
\begin{proof}
This follows directly  from Lemma \ref{exact sequence}, along with the fact that the symbol of $\diff_A^*$ is the adjoint of $w_\alpha$.
\end{proof}

Notice that we have dealt with the two parts of the infinitesimal action $R_A(u, \xi) = -\iota_u F_A - \diff_A\xi$ in different ways, because they are of different orders. Restriction to $V_A$ compensates for $\iota_uF_A$; adding $\diff_A^*$ compensates for~$\diff_A\xi$. 

\subsection{A Dirac operator}\label{dirac operator}

We next explain why, for a perfect connection $A$, the gauge-fixed linearisation $D_A$ is essentially a certain Dirac operator associated to $g_A$. We begin by giving a purely Riemannian description of $D_A$ in this case.

When $A$ is perfect, $F_A \colon E \to \Lambda^+$ is an isometry. We will use this freely throughout this section to replace $E$ by $\Lambda^+$. So $D_A$ is identified with an operator whose domain is sections of a sub-bundle $V \subset \Lambda^1 \otimes \Lambda^+$ and whose range is sections of $W = S^2_0\Lambda^+ \oplus \Lambda^+$. The sub-bundle $V$ is defined as the orthogonal complement of the homomorphism $S \colon TX \to \Lambda^1 \otimes \Lambda^+$ given by $S(u) = \iota_u F_A$. The next Lemma shows this has a Riemannian description.


\begin{lemma}
Identifying $\Lambda^1 \otimes \Lambda^+ \cong \Hom(TX, \Lambda^+)$, the map $S$ is determined by the condition that $S(u)(v)$ is metric dual to the self-dual bivector $(u \wedge v)^+$. 

It follows that at each point $x$ of $X$, the decomposition $\Lambda^1 \otimes \Lambda^+ \cong V \oplus TX$ is  $\SO(T_xX)$ invariant.
\end{lemma}
\begin{proof}
Given an orthonormal frame $e_1, e_2, e_3$ of $E$, write $F_A = \sum F_i \otimes e_i$. Then $S_A(u) = \iota_u F_A$ is identified with the element $\sum \iota_u F_i \otimes F_i$. So $S(u)(v)$ is metric dual to $ \sum F_i(u,v) F_i$. Since $A$ is perfect, the $F_i$ are an orthonormal basis for $\Lambda^+$ and the result follows.
\end{proof}

We next define a homomorphism 
\[
c \colon \Lambda^1 \otimes V \to W
\]
which will be the symbol of $D_A$ when $A$ is perfect. We begin with the map
\[
c_1 \colon \Lambda^1 \otimes \Lambda^1 \otimes \Lambda^+ \to S^2_0\Lambda^+
\]
given by the composition
\[
\left(\Lambda^1 \otimes \Lambda^1 \right)\otimes \Lambda^+
\to
\Lambda^2 \otimes \Lambda^+
\to
\Lambda^+ \otimes \Lambda^+ 
\to
S^2_0 \Lambda^+
\]
where the first arrow is skew-symmetrisation on the bracketed factor, the second arrow is projection to $\Lambda^+$ on the first factor and the third arrow is the projection determined by the metric onto the symmetric trace-free part. Next, we write $c_2$ for the homomorphism
\[
c_2 \colon \Lambda^1 \otimes \Lambda^1 \otimes \Lambda^+ \to \Lambda^+
\]
given by using the metric to contract the two copies of $\Lambda^1$. We now define 
\[
c = 2c_1 \oplus c_2  \colon \Lambda^1 \otimes V \to W
\] 
where we have restricted $c_1$ and $c_2$ to $\Lambda^1 \otimes V \subset \Lambda^1 \otimes \Lambda^1 \otimes \Lambda^+$.

\begin{lemma}\label{DA when A perfect}
When $A$ is perfect, the gauge-fixed linearisation 
\[
D_A \colon \Omega^0(V) \to \Omega^0(W)
\]
is the composition $D_A = c \circ \nabla$
\[
\Omega^0(V) \stackrel{\nabla}{\to} \Omega^1(V) \stackrel{c}{\to} \Omega^0(W)
\]
of the Levi-Civita connection and the homomorphism $c$.
\end{lemma}
 
\begin{proof}
When $A$ is perfect, Lemma \ref{linearisation of Q} gives that 
\[
\delta_AQ(a) = 2 L_A(\diff_Aa)_0
\]
where the 0-subscript denotes the trace-free part. Moreover, since $A$ is perfect, it follows from the definition of $L_A$ and the fact that $A$ is identified with the Levi-Civita connection $\nabla$ on $\Lambda^+$ that
\[
\delta_AQ = 2c_1 \circ \nabla
\]
Similarly, $\diff_A^* = c_2 \circ \nabla$ which completes the proof.
\end{proof}

We will now use this result to show that $D_A$ is essentially a coupled Dirac operator. Write $S_\pm$ for the spin bundles of $X$ and $S^m_\pm$ for the $m^{\text{th}}$~symmetric product of $S_\pm$. (We will only ever encounter even numbers of products of the $S_\pm$, so the question of whether or not $X$ is spin does not enter.) We will prove that $V \otimes \C = S^3_+ \otimes S_- $ and $W\otimes \C = S^3_+ \otimes S_+$ are opposite spin bundles and that $c$ is essentially Clifford multiplication between them. From this and Lemma \ref{DA when A perfect} it follows that $D_A$ is the Dirac operator coupled to the Levi-Civita connection on $S_+^3$ (modulo a certain choice of scaling, as will be explained). To describe this we begin with some standard results concerning $\Spin(4) = \SU(2)_+ \times \SU(2)_-$ representations pertinent to our discussion (all of which can be proved via the classification of representations of $\SU(2) \times \SU(2)$ and the Clebsch--Gordan formulae). In the following, $S_\pm$ is the fundamental representation of $\SU(2)_\pm$.

\begin{lemma}\label{spin reps}
There are the following isomorphisms of $\Spin(4)$-rep\-resen\-tations, where in each case the right-hand side is a decomposition into irreducible summands.
\begin{enumerate}
\item
$ \R^4 \otimes \C \cong S_+ \otimes S_-$.
\item
$\Lambda^+(\R^4) \otimes \C \cong S^2_+$.
\item
$S^2_0(S^2_+) \cong S^4_+$.
\item
$S^3_+ \otimes S_+ \cong S^4_+ \oplus S^2_+$
\item
$(S_+\otimes S_-) \otimes S^2_+ \cong (S_+ \otimes S_-) \oplus (S^3_+\otimes S_-)$.
\item
$(S_+ \otimes S_-) \otimes (S^3_+ \otimes S_-)
\cong
(S^4_+ \otimes S^2_-) \oplus (S^2_+\otimes S^2_-) \oplus S^4_+ \oplus S^2_+$.
\end{enumerate}
\end{lemma}

We now apply this to the tangent bundle of $X$. By irreducibility, the complexification of the splitting
\[
\Lambda^1 \otimes \Lambda^+ \cong V \oplus TX
\]
is given by part 5  of Lemma \ref{spin reps} (via parts 1 and 2). It follows that
\[
V \otimes \C = S^3_+ \otimes S_-.
\]
So the domain of $D_A$ is a twisted spin bundle. Moreover, by parts 2, 3 and 4 of Lemma~\ref{spin reps}, the range $W = S^2_0 \Lambda^+ \oplus \Lambda^+$ has complexification
\[
W \otimes \C
=
S^3_+ \otimes S_+
\]
which is the opposite spin bundle.

The bundle $S^3_+ \otimes S_-$ carries a Dirac operator 
\[
D^- \colon \Omega^0(S^3_+ \otimes S_-) \to \Omega^0(S^3_+ \otimes S_+)
\]
built by composing the Levi-Civita connection on $S^3_+$ with Clifford multiplication
\[
 (S_+ \otimes S_-)\otimes(S^3_+ \otimes S_-)
\to
S^3_+ \otimes S_+.
\]
In terms of Lemma \ref{spin reps}, this corresponds to projection on to the final pair of summands $S^4_+ \oplus S^2_+$ in the decomposition of part 6, which is the same as $S^3_+ \otimes S_+$ by part 4. Meanwhile, the complexification of the homomorphism $c$ is also a $\Spin(4)$-equivariant map between the same spaces. It follows from the decomposition of Lemma \ref{spin reps} and Schur's Lemma that there is a two-parameter family of such maps given by composing Clifford multiplication with scaling separately on $S^4_+$ and on $S^2_+$. Hence there is a 2-complex-parameter family of first-order elliptic differential operators $\Omega^0(S^3_+ \otimes S_-) \to \Omega^0(S^3 _+\otimes S_+)$ which include the Dirac operator $D^-$ and the complexification of~$D_A$. Note that neither scaling parameter can be zero (or ellipticity would fail) but that any other scale can be used. The family of operators is thus parametrised by the complement of the axes in $\C^2$. In particular it is connected

To a large extent these scale factors are irrelevant. From the point of view of ellipticity, the choice of relative scales of the summands in the definition $D_A = \delta_AQ \oplus \diff_A^*$ is arbitrary; the above argument makes it clear that a choice exists for which $D_A$ is exactly equal to the Dirac operator $D^-$ when $A$ is perfect.

\subsection{The index of $D_A$}\label{index}

\begin{proposition}
Given a definite connection $A$, the index of the gauge-fixed linearisation $D_A$ is
\[
\ind(D_A) = - 5\chi(X) - 7\tau(X),
\]
where $\tau(X)$ is the signature and $\chi(X)$  the Euler characteristic of $X$.
\end{proposition}

\begin{proof}
We have just seen that when $A$ is perfect, $D_A$ can be connected by a path of elliptic operators on $S^3_+ \otimes S_-$ to the Dirac operator $D^-$ coupled to the Levi-Civita connection on $S^3_+$. By deformation invariance of the index, $\ind D_A = \ind D^-$. 

In fact, this same equality holds for an arbitrary definite connection. This is because, whilst $D_A$ and $D^-$ are not simply related in general, their symbols are. More precisely, let $F_t \colon E \to \Lambda^+_A$ be a path of bundle-valued 2-forms, with $\ker F_t = 0$, $F_0 = F_A$ and $F_1$ an isometry onto $\Lambda_A^+$. Using $F_t$ in place of $F_A$ in the symbol sequence (\ref{symbol exact sequence}) gives a path of elliptic symbols starting at that of $D_A$ and ending at that of $D^-$. 
So, by deformation invariance of the index it suffices to compute $\ind(D^-)$. 

By the Atiyah--Singer index theorem this is
$$
\ind(D^-)
=
- \int_X \hat A \cdot \ch(S^3_+),
$$
where $\hat A$ is the A-hat genus of $X$, which for a four-manifold is $1 - \frac{1}{24} p_1(X)$. To compute the Chern character $\ch(S^3_+)$ first note that $S_+$ is an $\SU(2)$-bundle, so $\ch(S_+) = 2 -c_2(S_+)$ and $S_+ \otimes S_+ \cong \underline{\C} \oplus S^2_+$. It follows that
\[
\ch(S^2_+) 
\ =\ 
\ch(S_+) \cdot \ch(S_+) -1
\ =\
3 - 4c_2(S_+).
\]
Next we use the fact that $S_+ \otimes S^2_+ \cong S_+ \oplus S^3_+ $ to deduce
\[
\ch(S^3_+)
\ =\
\ch(S_+) \cdot ( \ch(S^2_+) - 1)
\ =\
4 - 10 c_2(S_+).
\]
So the index theorem says
\[
\ind(D^-) 
= 
\int _X \left( 10 c_2(S_+) + \frac{1}{6}p_1(X) \right).
\]
To compute this we use
\begin{eqnarray*}
\int_X p_1(X) &=& 3 \tau(X),\\
\int_X c_2(S_+) & =& -\frac{1}{4}(2\chi(X) + 3\tau(X)),\\
\end{eqnarray*}
This gives $\ind(D^-) = -5\chi(X) - 7 \tau(X)$.
\end{proof}

\subsection{A twistorial description of $D^-$ and rigidity}\label{twistor fact}

In this section we will explain why perfect connections are rigid modulo gauge; in other words, that if $A$ is perfect then $\ker D_A = 0$. We will also explain another phenomenon: in the cases of $S^4$ and $\overline{\C\P}^2$ the index is of $D_A$ is $-10$ and $-8$ respectively (the Fubini--Study metric on $\C\P^2$ is anti-self-dual in the non-complex orientation). In both cases this is precisely minus the dimension of the isometry group. However, for compact hyperbolic 4-manifolds, where the isometry group is discrete, the corresponding index is $-5 \chi(X)$ which is proportional to the volume and so never zero. 

The explanation of these fact comes from twistor theory and is certainly known to the experts. The main work is done in \cite{Hitchin1980Linear-field-eq,LeBrun1988A-rigidity-theo,Horan1996A-rigidity-theo} but since the whole statement we need is not explicitly to be found there we give here an overview of the argument, referring to the above articles for the proofs of various parts. All the material in this section was kindly explained to us by Claude LeBrun. 

To fix notation we recall some bare facts of twistor theory. For more information as well as proofs see \cite{Atiyah1978Self-duality-in,Besse1987Einstein-manifo}. Let $X$ be an oriented Riemannian manifold. The twistor space of $X$ is the projectivised spin bundle $Z= \P(S_+)$. Clifford multiplication defines a natural almost complex structure on $Z$ and this is integrable precisely when the metric on $X$ is anti-self-dual. In this case the fibres of $Z \to X$ are holomorphic $\C\P^1$s with normal bundle $\O(1)\oplus \O(1)$. Moreover, the anti-canonical bundle of $Z$ has a natural square root which we denote by $\O(2)$, since its fibrewise restriction is isomorphic to $\O(2) \to \C\P^1$. 

The first result we explain is the following.

\begin{theorem}[Cf.\ Hitchin \cite{Hitchin1980Linear-field-eq} Theorem 4.1(ii)]\label{twistor_cohom_dirac}
Let $X$ be an anti-self-dual manifold Riemannian four-manifold with twistor space $Z$. Write $D^-$ for the Dirac operator on $S_-\otimes S^3_+$ coupled to the Levi-Civita connection on $S^3_+$. Then:
\begin{enumerate}
\item $\ker D^- \cong H^1(Z,\O(2))$.
\item $\coker D^- \cong H^0(Z,\O(2)) \oplus H^2(Z,\O(2))$.
\end{enumerate}
\end{theorem}

The proof of this begins with a resolution of $\O(2)$. In fact, this resolution makes sense for an arbitrary holomorphic vector bundle $E \to Z$. (This procedure gives a nice way to carry out the Penrose transform in this context without recourse to the double fibration picture explained for example in \cite{Baston1989The-Penrose-tra}.) We write $\V(E)$ for the sheaf of smooth sections of $E$ which are holomorphic along the fibres of $Z \to X$. The first part of the resolution we are looking for is $0 \to \O(E) \to \V(E)$. 

To continue this sequence, consider the $\delb$-operator on $\V(E)\subset C^\infty(E)$. For an arbitrary smooth section $s$ of $E$, $\delb s \in C^\infty(\overline{T^*Z} \otimes E)$. However, when $s \in \V(E)$ then $\delb s$ vanishes on vertical $(0,1)$-vectors. Write $N \subset \overline{T^*Z}$ for those covectors vanishing on vertical $(0,1)$-vectors; $N \to X$ is a rank 2 complex bundle and by definition $\delb s$ is a section of $N \otimes E$. Now $N$ is not a holomorphic bundle over the whole of $Z$, however on restriction to each fibre it \emph{is} holomorphic: it is canonically identified with the normal bundle (as the notation is meant to suggest). This means we may define the sheaf of smooth sections of $N\otimes E$ which are vertically holomorphic. We denote this sheaf by $\V^{0,1}(E)$. If $s\in \V(E)$ then $\delb s \in \V^{0,1}(E)$, since $s$ is already vertically holomorphic. So we have extended our exact sequence of sheaves to $0 \to \O(E) \to \V(E) \to \V^{0,1}(E)$. 

To go one step further we consider the $\delb$-operator restricted to $\V^{0,1} \subset C^{\infty}(\overline{T^*Z} \otimes E)$ which takes values in $C^{\infty}(\Lambda^2\overline{T^*Z}\otimes E)$. One can show that if $s$ is a section in $\V^{0,1}(E)$ then the $(0,2)$-form $\delb s$ vanishes when one contracts with a single vertical $(0,1)$-vector. This amounts to the integrability of the $\delb$-operator together with the fact that the fibres of $Z \to X$ are complex curves. The subspace of $\Lambda^2\overline{T^*X}$ which vanishes on contraction with a single vertical $(0,1)$-vector is naturally identified with $\Lambda^2N$. So if $s$ is a section in $\V^{0,1}(E)$, then $\delb s$ is a smooth section of $\Lambda^2N \otimes E$. Again, $\Lambda^2N$ is not a holomorphic bundle on the whole of $Z$ but it \emph{is} holomorphic when restricted to a fibre, where it is naturally identified with the determinant of the normal bundle. This means we can consider the sheaf of smooth sections of $\Lambda^2N \otimes E$ which are fibrewise holomorphic, which we denote by $\V^{0,2}(E)$. Again, if $s$ is a section of  $\V^{0,1}(E)$ then $\delb s$ is a section of $\V^{0,2}(E)$. We now have an exact sequence of sheaves
\[
0 \to \O(E) \to \V(E) \to \V^{0,1}(E) \to \V^{0,2}(E) \to 0
\]
We would like to use this resolution to compute the cohomology of $\O(E)$. To this end we have the following result.

\begin{proposition}
If the restriction of $E$ to each fibre is a positive vector bundle then the resolution 
\[
0 \to \O(E) \to \V(E) \to \V^{0,1}(E) \to \V^{0,2}(E) \to 0
\]
is acyclic. 
\end{proposition}
\begin{proof}
There is a Leray--Serre spectral sequence coming from the fibration $Z \to X$ which computes the cohomology groups of $\V(E)$. The $E_2$-page is made up of the groups $H^p(X, \underline{H}^q(E))$ where $\underline{H}^q(E)$ is the sheaf of sections of the smooth vector bundle on $X$ whose fibre at $x$ is the $q^\text{th}$ cohomology group of the holomorphic bundle got by restricting $E$ to the fibre over~$x$. Since the sheaf $\underline{H}^q(E)$ over $X$ is soft, $H^p(X ,\underline{H}^q(E)) = 0$ for $p >0$. Since $E$ is fibrewise positive, $\underline{H}^q(E) = 0$ for $q>0$. It follows that the only non-zero element on the $E_2$-page is $H^0(X, \underline{H}^0(E))$. So the spectral sequence degenerates and the only non-zero cohomology group is $H^0(X, \V(E))$. 

We next run the same argument for the sheaf $\V^{0,1}(E)$. Recall that this is the sheaf of sections of $N \otimes E$ which are fibrewise holomorphic. Since the fibrewise restriction of $N$ is isomorphic to $\O(1) \oplus \O(1)$, the fibrewise restriction of $N\otimes E$ is again positive and so the only non-zero cohomology group is $H^0(X, \V^{0,1}(E))$.

Finally, we do the same for $\V^{0,2}(E)$. This sheaf is the sections of $\Lambda^2N \otimes E$ which are fibrewise holomorphic. The fibrewise restriction of $\Lambda^2N$ is isomorphic to $\O(2)$ and so $\Lambda^2N\otimes E$ is also fibrewise positive, completing the proof.
\end{proof}

We can now prove the result above relating the kernel and cokernel of $D^-$ to the cohomology of $\O(2)$.

\begin{proof}[Sketch of proof of Theorem \ref{twistor_cohom_dirac}]
$\O(2)$ is certainly fibrewise positive and so, by the previous result, we have that the cohomology of $\O(2)$ is equal to that of the complex
\[
H^0(Z, \V(2)) \stackrel{d_1}{\to} H^0(Z, \V^{0,1}(2)) \stackrel{d_2}{\to} H^0(Z, \V^{0,2}(2))
\]
Here, $\V(2) = \V(\O(2))$ and so forth. Now, $H^0(Z, \V(2))$ can be identified with $C^\infty(X, S^2_+)$. This is because $Z = \P(S_+)$ and moreover $S_+ \cong S_+^*$. In a similar vein, one can show that $H^0(Z, \V^{0,1}(2)) = C^{\infty}(X, S^3_+ \otimes S_-)$ and $H^0(Z, \V^{0,2}(2)) = C^\infty(X, S^4_+)$. See \cite{Hitchin1980Linear-field-eq} for details. Moreover, it is shown there that the maps $d_1$ and $d_2$ in the above complex combine to give the Dirac operator $D^-$, via the identification 4 of Lemma \ref{spin reps}:
\[
S^2_+ \oplus S^4_+ \cong S^3_+ \otimes S_+,
\]
\[
d_1^* + d_2 \cong D^-
\colon 
C^\infty(X, S^3_+ \otimes S_-)
\to
C^\infty(X, S^3_+ \otimes S_+)
\]
From this the result follows.
\end{proof}

We next add the hypothesis that, in addition to being anti-self-dual, $X$ is Einstein with non-zero scalar curvature. In this case we have the following theorem, due in the case of positive curvature to LeBrun \cite
{LeBrun1988A-rigidity-theo} and negative curvature to Horan \cite{Horan1996A-rigidity-theo}.

\begin{theorem}\label{rigidity}
Let $X$ be a compact oriented 4-manifold carrying an anti-self-dual Einstein metric of non-zero scalar curvature. Then $\ker D^- = 0$. In other words, perfect connections are infinitesimally rigid modulo gauge. Moreover, when the scalar curvature is positive, $\coker D^-$ is naturally identified with the complexification of the space of Killing fields of $X$.
\end{theorem} 

We briefly outline the argument behind this. Fixing a metric in the conformal class means that the fibres of $Z \to X$ have round metrics and using this one defines a Hermitian metric in $\O(2)$ (which is isomorphic to the vertical tangent bundle). The curvature $\omega$ of this metric is in general awkward to describe, but when the metric on $X$ is Einstein $\omega$ is non-degenerate. It is positive on the fibres, where it restricts to the area form. When the scalar curvature is positive $\omega$ is also positive transverse to the fibres, giving a Kähler form; when the scalar curvature is negative $\omega$ is negative transverse to the fibres. 

In the positive case this means that $\O(2) = K^{-1/2}$ is an ample bundle. Kodaira vanishing now implies that the higher cohomology groups of $\O(2)$ vanish and so $\coker(D^-) \cong H^0(Z, \O(2))$ whilst $\ker(D^-) = 0$. To relate $H^0(Z, \O(2))$ to the isometry group of $X$, we recall that the Einstein metric on $X$ translates into a holomorphic contact distribution $C \subset TZ$ on $Z$, with quotient $TZ/C \cong \O(2)$. Killing fields on $X$ lift to holomorphic contact vector fields on $Z$ and this sets up an isomorphism between the complexification of the space of Killing fields on $X$ and the space of holomorphic contact fields on $Z$. Meanwhile, the projection $TZ \to \O(2)$ identifies the space of contact fields with $H^0(Z, \O(2))$ (see, e.g., \cite{LeBrun1995Fano-manifolds-} for details of this identification). 

On the other hand when $X$ is anti-self-dual and of \emph{negative} scalar curvature the curvature of $\O(2)$ is an indefinite form of fixed type $(1,2)$. This immediately implies that there are no holomorphic sections of $\O(2)$: if there were a holomorphic section $s$, this curvature form would be $i\delb \del \log |s|$ (at least away from $s=0$) and so would be non-negative at a maximum of $|s|$. It also in fact implies that $H^1(Z,\O(2)) = 0$, although the argument here is more involved, hinging on Weitzenböck formulae. The upshot is then that when the scalar curvature is negative, $\ind(D^-) = -h^2(Z, \O(2))$.

\section{A flow for definite connections}

In this section we describe a flow which attempts to deform a given definite connection into a perfect connection. We also prove some initial results about this flow (short time existence, uniqueness and local stability) although the truly difficult work remains to be done (singularity formation, understanding obstructions to long time existence etc.).

\subsection{The gradient flow equation}\label{short time existence section}

We begin by recalling the definition of the \emph{energy} of a definite connection $A$ (Definition \ref{definition of energy} in the introduction):
\[
\E(A) = \int_X |Q(A)|^2 \mu(A).
\]
Proposition \ref{topological lower bound} says that $\E(A) \geq 8\pi^2p_1(E)$ with equality if and only if $A$ is perfect. In this section we will describe the downward gradient flow of $\E$. 

For this we need a metric on the space $\D$ of definite connections: recall that each $A \in \D$ defines a Riemannian metric $g_A$ on $X$ and hence an $L^2$-inner-product on $T_A \D = \Omega^1(X, \so(E))$. To give the equation describing the downward gradient flow we first recall the map $L_A \colon \Lambda^2 \otimes E \to S^2E$ defined above in equation (\ref{L}). We will also need the adjoint $L_A^* \colon S^2E \to \Lambda^2\otimes E$ of $L_A$ defined via $g_A$ and the fibrewise metric in $E$. To describe $L^*_A$, given $M \in \End(E)$ and $T \in \Lambda^2 \otimes E$ we abuse notation and write $M(T)$ for $(1\otimes M)(T) \in \Lambda^2\otimes E$, i.e., the result of applying $M$ to the $E$-factor of $T$ and leaving the 2-form part unchanged. With this understood we have:

\begin{lemma}\label{L*}
$L^*_A \colon S^2E \to \Lambda^2\otimes E$ is given by $L^*_A(M) = M(F_A)$.
\end{lemma}
\begin{proof}
To verify the formula let $T=\sum T_i\otimes e_i \in \Lambda^2 \otimes E$ and let $M \in S^2E$ be given by the matrix $(M_{ij})$ with respect to the basis $e_i$. Then
\[
(M,L_A(T)) = \tr (ML_A(T)) = \sum_{i,j} M_{ij}(T_i, F_j)
\]
Meanwhile, 
\[
(M(F_A),T) 
= 
\left(\sum_{i,j} M_{ij}F_j \otimes e_i, \sum_{k} T_k \otimes e_k\right)
=
\sum_{i,j} M_{ij}(F_j, T_i)
\]
\end{proof}

Let $B_A = 4Q(A) - \frac{2}{3}|Q(A)|^2\, \text{Id}$, a self-adjoint endomorphism of~$E$. We are now in a position to describe the gradient flow of $\E$.

\begin{proposition}\label{equation for gradient flow}
The downward gradient flow of $\E$ is given by
\begin{equation}\label{flow}
\frac{\del A}{\del t}
=
- \diff_A^* \left(B_A(F_A)\right),
\end{equation}
where $B_A(F_A)$ is the result of applying $B_A$ to the $E$-factor of $F_A$ and leaving the 2-form part unchanged. 
\end{proposition}

\begin{proof}
Let $a \in T_A\D$. Recall from equation (\ref{delta mu}) that the derivative of $\mu(A)$ in the direction $a$ is
\[
\delta_A \mu(a)
= \frac{2}{3} (F_A, \diff_A a) \mu(A).
\]
Meanwhile, Lemma \ref{linearisation of Q} says that the linearisation of $Q$ in the direction $a$ is
\[
\delta_A Q(a) = 2 L_A(\diff_Aa) - \frac{2}{3} (F_A, \diff_A a) Q(A).
\]
Hence,
\begin{eqnarray*}
\delta_A \E(a)
	&=&
	\int_X\left(
		2 \tr \left[ Q(A) \left(\delta_A Q(a)\right)\right]\mu(A) 
		+ \tr \left(Q(A)^2\right) \delta_A\mu(a)
		\right),\\
	&=&
	\int_X\left(
		4 \tr\left[ Q(A) L_A(\diff_Aa)\right] 
		-
		\frac{2}{3} (F_A, \diff_A a)\tr\left(Q(A)^2\right)
		\right) \mu(A),\\
	&=&
	\int_X \left(
		4 L_A^*(Q(A)) - \frac{2}{3} \tr\left(Q(A)^2\right)F_A, \diff_Aa
		\right) \mu(A).
\end{eqnarray*}
Since $L_A^*(Q(A)) = Q(A)(F_A)$ we see that $\delta_A\E(a)$ is the $L^2$-inner-product of $a$ with $\diff_A^* \left(B_A(F_A)\right)$ as claimed.
\end{proof}

It is interesting to compare this with the Yang--Mills flow for connections over a manifold with a \emph{fixed} Riemannian metric: 
\[
\frac{\del A}{\del t}  = -\diff_A^* F_A.
\]
The flow (\ref{flow}) is a sort of twisted version of the Yang--Mills flow.

\subsection{Short time existence and uniqueness}
\label{exists_unique}

Our first main result concerning the downward gradient flow is the following:

\begin{theorem}\label{short time existence}
Given a definite connection $A_0$, there is an $\epsilon >0$ and a path $A(t)$ of definite connections for $t \in[ 0,\epsilon)$ solving the downward gradient flow equation (\ref{flow}) with $A(0) = A_0$. The solution is unique for as long as it exists.
\end{theorem}

The proof of this theorem relies on a standard argument, known as ``DeTurck's trick'' in the context of Ricci flow (see \cite{DeTurck1983Deforming-metri} and also \S6.3 of \cite{Donaldson1990The-geometry-of} for the analogous approach to the Yang--Mills flow). The point is that the symmetries of the flow prevent it from being parabolic, but that this is the only extent to which parabolicity fails to hold. Considering a gauge-adjusted flow breaks the symmetry and gives a genuinely parabolic flow, for which short time existence and uniqueness is standard. Then paths of gauge transformations are used to pass between solutions to the parabolic flow and solutions of the original flow.

In our situation, DeTurck's approach takes the following form. The gauge group $\G$ acts on the space of definite connections $\D$ and preserves the vector field $v(A) = - \diff_A^*(B_A(F_A))$ generating the flow (\ref{flow}). Given $\eta \in \Lie(\G)$ and $A \in \D$, we write $R(\eta, A) \in T_A\D$ for the infinitesimal action of $\eta$ at $A$. DeTurck's trick is based on finding a well-chosen $\G$-equivalent flow: given a map $\xi \colon \D \to \Lie(\G)$, we define the ``$\xi$-adjusted flow'' by
\begin{equation}\label{adjusted_flow}
\frac{\del A}{\del t} = -\diff_A^*(B_A(F_A)) - R(\xi(A), A).
\end{equation}
The point is to choose $\xi$ so that (\ref{adjusted_flow}) is parabolic, which we will do shortly. 

There is a one-to-one correspondence between the solutions of the unadjusted and adjusted flows and which is based on the following Lemma. The proof is a simple calculation which we omit.
\begin{lemma}
~
\begin{enumerate} 
\item
Let $A(t)$ be a smooth path in $\D$ and $g(t)$ a smooth path in $\G$. Then
\[
\frac{\del}{\del t} (g(A))
=
g_*\left(\frac{\del A}{\del t}\right) + R\left(r(g^{-1})_* \left(\frac{\del g}{\del t}\right), g(A)\right)
\]
where $r(h) \colon \G \to \G$ is right multiplication by $h \in \G$.
\item
Given $g \in \G$, $\eta \in \Lie(\G)$ and $A\in \D$, 
\[
g_*(R(\eta,A)) = R(\Ad_g(\eta), g(A))
\]
where $\Ad_g$ denotes the adjoint action of $g$ on $\Lie(\G)$. 
\end{enumerate}
\end{lemma}

It follows from this Lemma, together with the $\G$-invariance of $\diff^*_A(B_A(F_A))$, that if $\hat{A}(t)$ is a solution to the $\xi$-adjusted flow (\ref{adjusted_flow}) then $A(t) =g(t)\hat{A}(t)$ solves the unadjusted flow (\ref{flow}) provided that $g(t)$ solves
\begin{equation}\label{from_adjutsed_to_unadjusted}
l(g^{-1})_*\left( \frac{\del g}{\del t}\right) = \xi(\hat A),
\end{equation}
where $l(g^{-1}) \colon \G \to \G$ denotes left multiplication by $g^{-1}$. Similarly, if $A(t)$ solves the unadjusted flow then $\hat A(t) = h(t) A(t)$ solves the $\xi$-adjusted flow provided that $h(t)$ solves
\begin{equation}\label{from_unadjusted_to_adjusted}
r(h^{-1})_* \left(\frac{\del h}{\del t}\right) = -\xi(h (A)).
\end{equation}

If we assume for a minute that $\xi$ has been chosen so that (\ref{adjusted_flow}) is parabolic, then a solution $\hat A$ to the adjusted flow with any starting point $A_0 \in \D$ exists for short time, by standard theory. (We make just such a choice of $\xi$ below.) Now (\ref{from_adjutsed_to_unadjusted}) is an ODE for $g$: regarding $g$ as a diffeomorphism of the principal frame bundle of $E$, (\ref{from_adjutsed_to_unadjusted}) simply says that $g(t)$ is generated by the time-dependent vector field $\xi(\hat A(t))$. Fixing $g(0) = \mathrm{id}$ we arrive at a solution $A(t) =g(t)\hat{A}(t)$ to (\ref{flow}) with $A(0) = A_0$.

The situation for (\ref{from_unadjusted_to_adjusted}) is a little more complicated, since now $h$ appears in the right-hand side of the equation as well, meaning that the equation is a PDE. We will choose $\xi$ so that the resulting equation (\ref{from_unadjusted_to_adjusted}) is a parabolic flow for $h$ and so short-time existence and uniqueness for $h(t)$ is guaranteed. (In the case of Ricci flow, this part of DeTurck's argument leads to the harmonic map flow.) Assuming this to be the case, we can now see that the solutions to (\ref{flow}) and (\ref{adjusted_flow}) are in one-to-one correspondence. Given $A(t)$ solving (\ref{flow}) we solve first for $h$, with $h(0)= \mathrm{id}$, and then for $g$ with $\hat A = h(A)$ and $g(0) = \mathrm{id}$  to obtain a second solution $A' = g(h(A))$ to (\ref{flow}). Now
\[
\frac{\del}{\del t} (g h)
=
r(h)_* \left(\frac{\del g}{\del t} \right) + l(g)_*\left(\frac{\del h}{\del t} \right)
=
0,
\]
by virtue of (\ref{from_adjutsed_to_unadjusted}) and (\ref{from_unadjusted_to_adjusted}) and so $gh = \mathrm{id}$ and $A'=A$. Similarly starting from a solution to (\ref{adjusted_flow}) and solving first for $g$ and then for $h$ one returns again to the same solution of the adjusted flow. Assuming then that (\ref{from_unadjusted_to_adjusted}) is parabolic, this completes the proof that solutions to the flows are in one-to-one correspondence. So once we prove that there is a unique short-time solution to the flow (\ref{adjusted_flow}) with given starting point $A_0$ it will follow that there is a unique short-time solution to the flow (\ref{flow}) starting at $A_0$.

We now turn to the details of implementing DeTurck's trick. Given $A_0 \in \D$, we write $A = A_0 + a$ for $a \in \Omega^1(X, \so(E)$ and set
\[
\xi(A) = - \diff_A^* a - \chi_A\left(S_A^* \diff_A^*\diff_A a\right)
\]
Here, $S_A \colon TX \to \Lambda^1 \otimes \so(E)$ is the map $S_A(u) = \iota_u F_A$ whilst $S_A^*$ is defined via the inner-products on $TX$ and $\Lambda^1\otimes \so(E)$ determined by $g_A$. Meanwhile, $\chi_A \colon \Omega^0(X, TX) \to \Lie(\G)$ is the horizontal lift with respect to $A$. Recalling the formula (\ref{infinitesimal action}) for the infinitesimal action $R(\xi, A)$ the flow (\ref{adjusted_flow}) is 
\[
\frac{\del A}{\del t}
=
- \diff_A^*(B_A(F_A)) 
- \diff_A\diff_A^* (a)
- S_AS_A^*(\diff_A^*\diff_A a).
\]

We now show that this adjusted flow is parabolic. The first step is to describe the linearisation of the original gradient flow. To do do this, we need to introduce some more notation. Let $\Pi_A \colon S^2E \to S^2E$ denote orthogonal projection onto $\langle Q(A) \rangle^\perp$, the self-adjoint endomorphisms orthogonal to $Q(A)$. Write $\Theta_A \colon S^2E \to S^2E$ for the map
\[
\Theta_A(M) = 8M + \frac{8}{9} \tr \left(Q(A)^2\right) \tr M.
\]
$\Theta_A$ is a positive-definite self-adjoint operator and so has a well-defined square-root $\Theta_A^{1/2}$. Finally, we define \begin{equation}\label{P}
P_A 
= 
\Theta_A^{1/2} \circ \Pi_A \circ L_A 
\colon 
\Lambda^2 \otimes E \to S^2 E. 
\end{equation}

\begin{proposition}\label{principal part}
The principal part (i.e., second order part) $\mathcal L_A$ of the linearisation of the right-hand-side of the gradient flow (\ref{flow}) of $\E$ at the point $A \in \D$ is given by
\[
\mathcal L_A = - \diff_A^* P_A^* P_A \diff_A \colon \Omega^1(X, \so(E)) \to \Omega^1(X, \so(E)).
\]
\end{proposition}

\begin{proof}
We must linearise the map $A \mapsto - \diff_A^*(B_A(F_A))$. Let $a \in \Omega^1(X, \so(E))$ denote the direction in which we will linearise. First we note that since $F_A$ is self-dual, $B_A(F_A)$ is self-dual. Hence $\diff_A^*(B_A(F_A)) = *\, \diff_A (B_A(F_A))$. 

Next we apply the Leibniz law in the following form. Given bundle-valued forms $S \in \Omega^p(X, \End(E))$ and $T \in \Omega^q(X, E)$ we write $M \wedge T \in \Omega^{p+q}(X,E)$ for the form obtained by tensoring the wedge product on forms with the natural action of $\End(E)$ on $E$. The Leibniz law reads
\begin{equation}\label{leibniz}
\diff_A \left(M \wedge T\right) = \diff_A M \wedge T \pm M \wedge \diff_A T
\end{equation}
 where the sign depends on the degree of $M$. It follows from this and the Bianchi identity $\diff_AF_A = 0$ that $\diff_A(B_A(F_A)) = (\diff_A B_A )\wedge F_A$. So we must linearise the map $A \mapsto -\ast \left((\diff_A B_A) \wedge F_A \right)$,
 
The Hodge star depends algebraically on $g_A$ and so is first order in $A$. This means that it does not contribute to the principal part of the linearisation. So we are seeking the second-order contribution of
\[
-* \delta_a \left( (\diff_A B_A) \wedge F_A\right)
=
-*\big( a\wedge B_A(F_A) + (\diff_A \delta_AB(a) )\wedge F_A + (\diff_A B_A )\wedge \diff_Aa\big).
\]
In this equation the only term which contributes to the principal part is 
\[
\mathcal L_A(a) = -* \big((\diff_A \delta_AB(a)) \wedge F_A\big).
\]
 
Next we differentiate $B_A= 4Q(A) - \frac{2}{3}\tr \left(Q(A)^2\right)$ using the equation for $\delta_AQ(a)$ from Lemma \ref{linearisation of Q} to obtain:
\begin{eqnarray*}
\delta_A B(a) 
	&=&
	4\delta_A Q(a) - \frac{4}{3} \tr\left[Q(A) (\delta_AQ(a))\right],\\
	&=&
	8L_A(\diff_A a) 
	- \frac{8}{3}(F_A, \diff_Aa) Q(A)
	- \frac{8}{3}\tr\left[Q(A)L_A(\diff_Aa)\right] \\
	&&
	\quad\quad
	+ \frac{8}{9}(F_A, \diff_Aa)\tr\left(Q(A)\right)^2.
\end{eqnarray*}
A straightforward calculation using
\[
\Pi_A(M) = M - \frac{\tr(Q(A)M)}{\tr(Q(A))^2} Q(A)
\]
(and recalling that $\tr (L_A(T)) = (F_A, T)$) shows that this is the same as
\begin{equation}\label{delta B}
\delta_A B(a) = \Pi_A \Theta_A \Pi_A \left(L_A(\diff_A a)\right).
\end{equation}

Denote by $\Psi_A \colon \Lambda^1 \otimes S^2E \to \Lambda^1 \otimes E$ the homomorphism
\[
\Psi_A(C) = * (C \wedge F_A).
\]
With this notation,
\[
\mathcal L_A(a) = - \Psi_A\left(\diff_A\left(\Pi_A \Theta_A \Pi_A \left(L_A(\diff_Aa)\right)\right)\right).
\]
We now claim that $\Psi_A \circ \diff_A = \diff_A^* \circ L_A^*$. To verify this, let $M \in \Omega^0(X, S^2E)$. Then $L_A^*(M) = M(F_A)$ by Lemma \ref{L*}. Since $M(F_A)$ is a section of $\Lambda^+\otimes E$, it follows that $\diff_A^*(M(F_A)) = *\, \diff_A(M(F_A))$. Now the Leibniz law (\ref{leibniz}) and the Bianchi identity $\diff_A F_A = 0$ imply this is equal to $*((\diff_A M) \wedge F_A) = \Psi_A(\diff_A M)$ as claimed. Using this we can finally write $\mathcal L_A(a) = -\diff_A^* P_A^*P_A \diff_A(a)$.
\end{proof}

We will now show how the symbol of $\mathcal L_A$ arises from a short exact sequence. Before giving the sequence, we need an identity from 4-dimensional Riemannian geometry.

\begin{lemma}\label{4d lemma}
On an oriented 4-dimensional Riemannian manifold $X$, let $\alpha \in T^*X$, $u \in TX$ and $\beta_1, \beta_2 \in \Lambda^+$. Then,
\[
(\alpha \wedge \iota_u \beta_1, \beta_2) 
+ 
( \alpha\wedge \iota_u \beta_2, \beta_1)
=
(\beta_1, \beta_2) \alpha(u).
\]
\end{lemma}
\begin{proof}
We prove the equation multiplied by the volume form, i.e., 
\[
\alpha\wedge \iota_u\beta_1 \wedge \beta_2
+
\alpha \wedge \beta_1 \wedge \iota_u \beta_2
-
\iota_u \alpha \wedge \beta_1 \wedge \beta_2
=
0.
\]
But this follows from applying $\iota_u$ to the necessarily zero 5-form $\alpha \wedge \beta_1 \wedge \beta_2$.
\end{proof}

\begin{corollary}\label{L circ S}
Let $A$ be a definite connection. For any $\alpha \in \Lambda^1$, $u \in TX$, 
$$L_A(\alpha \wedge \iota_u F_A) = \alpha(u) Q(A).$$
\end{corollary}
\begin{proof}
Recall that, for $b = \sum b_i \otimes e_i \in \Lambda^1 \otimes E$, $L_A(\alpha \wedge b)$ is given by the matrix
\[
L_A(\alpha \wedge b)_{ij}
=
\frac{1}{2} 
(\alpha\wedge b_i, F_j)
+
\frac{1}{2}
(\alpha \wedge b_j, F_i)
\]
By Lemma \ref{4d lemma}, it follows that $L_A(\alpha \wedge \iota_uF_A)_{ij} = (F_i, F_j)\alpha(u)$.
\end{proof}

\begin{remark}\label{direct calc}
We pause to make an aside. In the proof of ellipticity in \S\ref{elliptic}, a key part of the proof of Lemma \ref{exact sequence} was the fact that the symbol $\sigma_A(\alpha)$ of $\delta_AQ$ in the direction $\alpha$ satisfies $\sigma_A(\alpha) \circ S_A= 0$. At the time this was deduced on general grounds from $\G$-equivariance. With Corollary \ref{L circ S} in hand, it can also be proved directly by a calculation from the equation (\ref{formula for symbol}).
\end{remark}

We return to the discussion of the symbol of $\mathcal L_A$. Given $\alpha \in T^*X$, we write $w_\alpha \colon \Lambda^1 \otimes E \to \Lambda^2 \otimes E$ for the homomorphism given by the wedge-product with $\alpha$. Let $H_\alpha \colon \Lambda^1 \otimes E \to S^2E$ denote the composition $P_A \circ w_\alpha$. ($P_A$ is defined in equation (\ref{P}).) By Proposition \ref{principal part},  $-H_\alpha^* H_\alpha$ is the symbol of $\mathcal L_A$ in the direction $\alpha$. 

We write $V_\alpha = \im H_\alpha \subset S^2E$. We claim that $\dim V_\alpha = 5$. To see this, recall that $P_A=\Theta_A^{1/2} \Pi_A L_A$. Now, $L_A \colon \Lambda^2 \otimes E \to S^2E$ is surjective when restricted to $\Lambda^+\otimes E$ (Lemma \ref{matrix form of L}). Since the composition of $w_\alpha \colon \Lambda^1 \to \Lambda^2$ with projection to $\Lambda^+$ is also surjective it follows that $L_A \circ w_\alpha$ is surjective. Now, since $\Pi_A$ is projection onto a subspace of dimension 5 and $\Theta$ is invertible, it follows that $V_\alpha$ has dimension 5.

\begin{lemma}\label{parabolic exact sequence}
Fix $\alpha \neq 0$. For any $A \in \D$ and at every point of $X$, there is a short exact sequence
\[
\xymatrixcolsep{2.5pc}
\xymatrix{
0
\ar[r]  
&TX \oplus E 
\ar[r]^-{S_A+w_\alpha}
&\Lambda^1 \otimes E
\ar[r]^-{H_\alpha}
&V_\alpha
\ar[r]
&0}.
\]
It follows that for any $\alpha \neq 0$, if given $a\in \Lambda^1\otimes E$, all of $H_\alpha(a)$, $S_A^*(a)$ and  $w_\alpha^*(a)$ vanish, then $a = 0$.
\end{lemma}
\begin{proof}
We have already seen in the course of the proof of Lemma \ref{exact sequence}  that $S_A + w_\alpha$ is injective. Meanwhile $H_\alpha \circ w_\alpha = 0$ because $\alpha\wedge\alpha=0$. Moreover, Corollary \ref{L circ S} says that $\Pi_A \circ L_A \circ w_\alpha \circ S_A = 0$, hence $H_\alpha \circ S_A = 0$.
\end{proof}

We are now ready to prove that (\ref{adjusted_flow}) is parabolic.

\begin{theorem}\label{adjusted_flow_is_parabolic}
The flow (\ref{adjusted_flow}), i.e., the flow for $A(t) = A_0 + a(t)$ given by
\[
\frac{\del A}{\del t}
=
-\diff_A^*(B_A(F_A)) - \diff_A \diff_A^*a - S_AS^*_A(\diff_A^*\diff_A a)
\]
is strongly parabolic for $a$ sufficiently near to $0$. 
\end{theorem}
\begin{proof}
We must linearise the right hand side of the flow at $a=0$ and show that it is elliptic with negative definite symbol. The first term is the downward gradient flow of $\E$ and so contributes $\mathcal L_A = -\diff_A^* P^*_AP_A\diff_A$ to the principal part of the linearisation. The remaining terms in (\ref{adjusted_flow}) are straightforward to linearise at $a=0$. For example, when linearising the term $-\diff_A\diff_A^*a$ in the direction $b$, the only contribution is $-\diff_A\diff_A^*b$, because the infinitesimal change in, say $\diff_A^*$ is then evaluated on $a=0$ and so vanishes. Similarly, the linearisation of $-S_AS^*_A(\diff_A^*\diff_A a)$ is $-S_AS^*_A(\diff_A^*\diff_A b)$. Hence the principal part of the overall linearisation at $a=0$ is
\[
- \diff_A^*(P_A^*P_A(\diff_Ab)) - S_AS^*_A(\diff_A^*\diff_A b) - \diff_A\diff_A^* b
\]
The symbol of this in the direction $\alpha \in \Lambda^1$ is
\[
\Sigma(\alpha)
=
- \left(
H_\alpha^* H_\alpha 
+ 
S_AS_A^*  w^*_\alpha w_\alpha
+
w_\alpha w^*_\alpha\right)
\colon 
\Lambda^1\otimes E \to \Lambda^1 \otimes E.
\]
We must show that $\Sigma(\alpha)$ is negative definite for all non-zero $\alpha$.

To proceed, we split $\Lambda^1 = \langle \alpha \rangle \oplus \langle \alpha \rangle ^\perp$. For $\alpha$ of unit length, $w_\alpha w_\alpha^*$ is orthogonal projection onto the first summand, whilst $w^*_\alpha w_\alpha$ is orthogonal projection on to the second. Let $b \in \Lambda^1 \otimes E$ and write $b = \alpha \otimes x + c$ where where $ c \in \langle \alpha \rangle ^\perp \otimes E$ and $x \in E$. Then taking the inner-product with $b$ gives
\[
-(\Sigma(\alpha)(b),b)
=
|H_\alpha(c)|^2
+
|S_A^*c|^2
+
(S_AS_A^*c, \alpha \otimes x)
+
|x|^2.
\]
We write locally $F_A = \sum F_i \otimes e_i$ where $e_i$ are an orthonormal frame of $E$. For $u \in TX$,
\[
|S_A(u)|^2 = \sum |\iota_u F_i|^2 = \sum|F_i|^2 |u|^2 = 3 |u|^2.
\]
(Here we use the fact that for a unit-length self-dual 2-form $\theta$, the map  $u \mapsto \iota_u\theta$ is an isometry $TX \to T^*X$.) So $|S_AS_A^*c| = \sqrt{3} |S_A^*c|$. It follows that
\begin{eqnarray*}
|S_A^*c|^2 + (S_AS_A^*c, x\alpha) + |x|^2
	&\geq&
		|S_A^*c|^2 - \sqrt 3 |S_A^*c||x| + |x|^2\\
	&\geq& 
		\left(1 - \frac{\sqrt 3}{2}\right) \left(|S_A^*c|^2 + |x|^2\right).
\end{eqnarray*}
Hence 
\[
-(\Sigma(\alpha)(b),b)
\geq
|H_\alpha (c)|^2 + \left(1 -\frac{\sqrt 3}{2}\right)\left(
|S_A^*(c)|^2 + |x|^2 \right).
\]
This is non-negative and $\Sigma(\alpha)(b)=0$ implies that all of $H_\alpha(c)$, $S_A^*(c)$ and $x$ vanish. I.e., $H_\alpha(b)=0$, $S_A^*(b)= 0$ and $w_\alpha^*(b) = 0$. By Lemma \ref{parabolic exact sequence} it follows that $b=0$ and $\sigma(\alpha)$ is negative definite as required.
\end{proof}

We can now appeal to the standard theory of parabolic PDEs which says that for any starting data $A_0$, a solution to (\ref{adjusted_flow}) exists for short time and, moreover, is unique for as long as it exists (see, for example, Proposition 8.1 in \cite{Taylor2011Partial-differe}). Now applying a path of gauge transformations solving (\ref{from_adjutsed_to_unadjusted}) proves short-time existence of (\ref{flow}).

It remains to check uniqueness of the solution to (\ref{flow}) which in terms of DeTurck's approach comes down to proving that (\ref{from_unadjusted_to_adjusted}) is a parabolic flow for the path of gauge transformations $h(t)$. Explicitly, we assume that $A(t)$ solves (\ref{flow}) starting at $A_0$ and consider the following flow for $h(t) \in \G$:
\begin{equation}\label{uniqueness_flow}
r(h^{-1})_*\left(\frac{\del h}{\del t} \right)
=
\diff_{\hat{A}}^*a + \chi_{\hat{A}}(S^*_{\hat{A}}\diff_{\hat{A}}^*\diff_{\hat{A}} a)
\end{equation}
with $h(0) = \mathrm{id}$, where $\hat{A} = h(A)$ and $a = \hat{A} - A_0$. (Recall that $\chi_{\hat{A}} \colon \Omega^0(X,TX) \to \Lie(\G)$ is the horizontal lift with respect to $\hat{A}$.)

At first sight this appears to be a third order flow: $h(A)$ and hence $a$ is first order in $h$ and this is then differentiated twice in the second term on the right-hand side of (\ref{uniqueness_flow}). However, two of these derivatives combine to give the curvature tensor and so the flow is actually second order. 

To see that it is parabolic, we linearise at $h = \mathrm{id}$ in the direction of an infinitesimal gauge transformation $\eta = \phi + \chi_{A_0}(u)$ where $\phi \in \Omega^0(X, \so(E))$ and $u \in \Omega^0(X,TX)$. Since $a = 0$ at the point at which we are linearising, we obtain the following for the linearisation of the right-hand side:
\[
\eta \mapsto
-\diff_{A_0}^*
\left(\diff_{A_0}\phi + S_{A_0}(u)\right)
-
\chi_{A_0}\left(
S_{A_0}^* \diff^*_{A_0}\diff_{A_0} 
\left(\diff_{A_0}\phi + S_{A_0}(u)\right)
\right)
\]
(Here we have used (\ref{infinitesimal action}) which gives the infinitesimal action of $\eta$ on $A_0$ and hence $a$.) As remarked above, the seemingly third order term becomes
\[
S_{A_0}^*\diff_{A_0}^*\diff_{A_0}^2 \phi
=
S_{A_0}^* \diff_{A_0}^*(F_{A_0} \phi)
\]
which is first order in $\phi$. The leading order piece is thus second order and is given by
\[
-\diff_{A_0}^*\diff_{A_0} \phi
- \chi_{A_0} S_{A_0}^*\diff_{A_0}^*\diff_{A_0} S_{A_0}(u)
\]
which has symbol in the direction $\alpha$
\[
\Sigma(\alpha)(\eta) 
=  
- |\alpha|^2 \left(\phi + \chi_{A_0} S_{A_0}^*S_{A_0}u\right)
\]
Since $S_{A_0}$ and $\chi_{A_0}$ are injections, this is negative definite as required. So (\ref{uniqueness_flow}) is indeed a parabolic flow. This completes the proof of uniqueness of the solution (\ref{flow}) and hence the proof of Theorem \ref{short time existence}.

\subsection{Local stability of the flow}\label{local_stability_section}

In this section we prove local stability of the flow: if $\hat{A}$ is a perfect connection and $A_0$ is sufficiently close to $\hat{A}$ then the flow starting at $A_0$ exists for all time and converges, modulo gauge, exponentially fast to $\hat{A}$. More precisely:

\begin{theorem}\label{local_stability_thm}
Let $\hat{A}$ be a perfect connection over a compact four-manifold $X$ and let $k\geq 3$ be an integer. Then there exists $\epsilon > 0$ and $\alpha>0$ such that if $A_0$ is another definite connection with $\|\hat{A} - A_0\|_{L^2_{k}} < \epsilon$ then 
\begin{enumerate}
\item
The downward gradient flow $A(t)$ starting at $A_0$ exists for all time.
\item
There exist gauge-transformations $g(t) \in \G$ such that 
\[
\|g(t)^*A(t) - \hat{A} \|_{L^2_{k}} = o(e^{-\alpha t})
\]
\end{enumerate}
(We use a fixed background metric, e.g. that determined by $\hat{A}$, to define the Sobolev norms here.)
\end{theorem}

To prove this we first consider a gauge-fixed flow and prove a local stability result there. In fact, rather than use the same gauge-fixed flow from before, it turns out to be simpler to use the following flow:
\begin{equation}\label{parabolic_stable_flow}
\frac{\del A}{\del t}
= 
-\diff_A^*(B_A(F_A))
-
\diff_A\diff_A^*a
-
S_A(\Delta_A S^*_A(a))
\end{equation}
where $A(t) = \hat{A}+a(t)$ and $\Delta_A$ is the Laplacian on sections of $TX$ defined by the metric $g_A$. The proof that (\ref{parabolic_stable_flow}) is parabolic is essentially identical to that of Theorem \ref{adjusted_flow_is_parabolic} and so we it. By solving (\ref{from_adjutsed_to_unadjusted}) for a path of gauge transformations we then convert the solution to (\ref{parabolic_stable_flow}) into a long time solution to (\ref{flow}). (The reason not to have used (\ref{parabolic_stable_flow}) above is that (\ref{from_unadjusted_to_adjusted}), which was used to prove uniqueness, is not parabolic with this choice of $\xi$, rather it is a third order flow. This situation is avoided with the previous choice of $\xi$)

The key to proving local stability of (\ref{parabolic_stable_flow}) is to show that when $A$ is perfect the linearisation of the right-hand-side of (\ref{parabolic_stable_flow}) at $A$ is a \emph{definite} operator.

\begin{proposition}
When $A$ is perfect the operator
\[
G_A= \diff_A^*(P_A^*P_A(\diff_Ab)) + S_A(\Delta_A(S^*_A b)) + \diff_A\diff_A^* b
\]
is a self-adjoint, positive-definite, strongly elliptic operator.
\end{proposition}
\begin{proof}
The operator is manifestly self-adjoint and nonnegative; moreover the  proof that it is strongly elliptic is near identical to that of Theorem \ref{adjusted_flow_is_parabolic} and so we omit the details. The point is to prove that it is positive definite. 

Suppose $G_A(b) = 0$. Then, taking the inner-product with $b$ we see that
\[
P_A(\diff_A b) = 0,
\quad
\nabla_AS^*_A(b) = 0,
\quad
\diff_A^*b = 0.
\]
First note that the equation $\nabla_AS^*_A(b) = 0$ tells us that $S^*_A(b) = v$ is a covariant constant vector field. Either $v$ vanishes identically or it is nowhere zero. If $v$ is nowhere zero, then $\chi(X) = 0$. Moreover, interior contraction with $v$ defines isomorphisms $\Lambda^+ \cong \langle v \rangle^0 \cong \Lambda^-$ which, since $v$ is covariant constant, also match up the Levi-Civita connections on $\Lambda^+$ and $\Lambda^-$. It follows that the metric is both self-dual and anti-self-dual, hence conformally flat and so $\tau(X) = 0$. But this contradicts the fact that $2\chi(X) + 3\tau(X) > 0$. So we must actually have $S_A^*(b) =0$.

We will now show that the remaining two equations on $b$ mean that $b$ determines a section of $S_-\otimes S_+^3$ satisfying the Dirac equation. We begin by identifying, as always, $\Lambda^1\otimes E \cong \Lambda^1 \otimes \Lambda^+$, via the identification $E \cong \Lambda^+$ given by $F_A$. Recall from Lemma \ref{spin reps} that
\[
\Lambda^1 \otimes \Lambda^+ 
\cong 
V \oplus TX
\]
where $V \otimes \C = S_- \otimes S^3_+$. Moreover, $S_A^*$ is identified with projection onto $S_- \otimes S_+ = TX \otimes \C$. Since $S^*_A(b)=0$ we see that $b$ is actually just a section of $S_-
\otimes S_+^3$. Now recall the definition (\ref{P}) of $P_A$, 
\[
P_A = \Theta^{-1/2}\circ \Pi_A \circ L_A
\]
In our current situation, where $A$ is perfect, $L_A \colon \Lambda^2 \otimes \Lambda^+$ is the natural projection
\[
\Lambda^2 \otimes \Lambda^+ 
\to
\Lambda^+ \otimes \Lambda^+
\to 
S^2(\Lambda^+)
\]
Meanwhile, $\Pi_A$ is the projection $S^2(\Lambda^+) \to S^2_0(\Lambda^+)$ and $\Theta^{-1/2}$ is multiplication by a non-zero constant on $S^2_0(\Lambda^+)$. It follows that $P_A(\diff_A b) = 0$ if and only if $b \in C^\infty(V)$ maps to zero under the map 
\[
C^\infty(V) \to C^\infty(S^2_0\Lambda^+)
\]
But from the discussion in \S\ref{dirac operator} this is simply one component of the Dirac operator, whilst $\diff_A^*$ is the other. Hence $D^-(b) = 0$.

Finally, we invoke Theorem \ref{rigidity}, which tells us that $\ker D^- =0$. Hence $G_A(b) = 0$ if and only if $b=0$ and so $G_A$ is positive definite as claimed.
\end{proof}

With this result in hand, local stability for the gauge-fixed flow (\ref{parabolic_stable_flow}) is essentially standard. We give the details for want of a specific reference.

\begin{theorem}
Let $\hat{A}$ be a perfect connection over a compact four-manifold $X$ and let $k \geq 3$ be an  integer. Then there exists $\epsilon >0$ and $\alpha>0$ such that if $A_0$ is another definite connection with $\|\hat{A}-A_0\|_{L^2_k}< \epsilon$ then
\begin{enumerate}
\item
The gauge-fixed flow (\ref{parabolic_stable_flow}) $A(t)$ starting at $A_0$ exists for all time.
\item
$\|A(t) - \hat{A}\|_{L^2_{k}} = o(e^{-\alpha t})$.
\end{enumerate}
(We use a fixed background metric here, e.g.\ the one determined by $A$, to define the Sobolev norms.)
\end{theorem}
\begin{proof}
The proof is based on the implicit function theorem. We begin by setting up the appropriate Banach spaces. We write $P_{2k}$ for the completion of the space of compactly supported sections of $\Lambda^1 \otimes E$ defined over $X \times [0, \infty)$ with respect to the norm
\[
\| a \|^2_{P_{2k}}
=
\sum_{j=0}^{k}
\int_0^\infty e^{2\alpha t} \| \del_t^ja(t) \|^2_{L^2_{2(k-j)}} \, \diff t
\]
Here $a(t)$ is the section over the slice $X \times\{t\}$. The parameter $\alpha >0$ will be chosen in what follows. Note that, since $\alpha >0$, we are imposing exponential decay in the spatial norms as $t \to \infty$. Note also that we are weighting $\del_t$ as ``second order'', in accordance with the second order parabolic equation we are studying.

It is standard that the evaluation map $a \mapsto a(0)$, defined on smooth compactly supported sections over $X \times [0,\infty)$ extends to a continuous map $P_{2k} \to L^2_k$. With this in hand, we define a map $\Phi \colon P_{2k} \to P_{2k-2} \times L^2_{k}$ by
\[
\Phi(a) = \left(
\frac{\del a}{\del t}
+
\diff_A^*(B_A(F_A))
+
\diff_A\diff_A^*a
+
S_A(\Delta_A S^*_A(a)),
a(0) 
\right)
\]
where $A(t) = \hat{A} + a(t)$. The choice of $k \geq 3$ ensures that we are in a range where Sobolev multiplication holds and so $\Phi$, defined a~priori on smooth compactly supported sections, extends to a smooth map between the stated Banach spaces. Note that $\Phi(0) = (0,0)$. We want to show that for $a_0$ sufficiently small, there is a solution to $\Phi(a) = (0, a_0)$.

To do this we use the implicit function theorem. The linearisation of $\Phi$ at zero is 
\[
D\Phi(b) 
= 
\left(\frac{\del b}{\del t} + G_{\hat{A}}(b),b(0)\right)
\]
Since $\hat{A}$ is perfect, $G_{\hat{A}}$ is a positive-definite self-adjoint strongly elliptic operator, it is standard that for $\alpha>0$ sufficiently small, $D\Phi$ is an isomorphism of Banach. It now follows from the implicit function theorem that for $a_0$ sufficiently small in $L^2_k$, there exists a solution $a$ to $\Phi(a) = (0,a_0)$. By a standard parabolic regularity argument, $a$ is in fact smooth and $A(t) = \hat{A} + a(t)$ gives the sought-after long-time solution to the flow~(\ref{parabolic_stable_flow}). 

It remains to show that 
\[
\|A(t) - \hat{A}\|_{L^2_{k}} = \|a(t)\|_{L^2_{k}} =  o(e^{-at}).
\]
To see this, let $f(t) = e^{\alpha t}\|a(t)\|_{L^2_k}$. Since $a \in P_{2k}$ it follows that $f \in L^2_1[0,\infty)$. Now by Sobolev embedding there is a constant $C$ such that for any $t$, 
\[
| f(t)| \leq C \| f\|_{L^2_1[t-1,t+1]}
\]
Since the right-hand side tends to zero as $t \to \infty$, we see that $f(t) \to 0$ as $t \to \infty$, as claimed.
\end{proof}

The proof of Theorem \ref{local_stability_thm} follows immediately from this upon solving (\ref{from_adjutsed_to_unadjusted}) to convert the solution to (\ref{parabolic_stable_flow}) back into a solution to the original flow (\ref{flow}).

\section{Energy and a flow for definite triples}\label{flow for definite triples}

There is also an energy function and associated flow for the definite triples of symplectic forms which appear in Conjecture \ref{skd_conjecture}. As we will see, the theory parallels  that of $\E \colon \D \to \R$. There is one important additional result available in this context, namely Proposition \ref{unique critical point} below, which says that in the case of definite triples, the only possible critical points of the energy function are hyperkähler triples. We do not know if the analogue of this result holds for definite connections.

We begin by recalling the definition of a definite triple.

\begin{definition}
A \emph{definite triple} on a 4-manifold $X$ is a triple of symplectic forms $\omega_1, \omega_2, \omega_3$ which span a definite 3-plane in $\Lambda^2$ at each point of $X$. 
\end{definition}

We would like to deform a given definite triple to a cohomologous hyperkähler triple. Of course, for this to be possible, the corresponding cohomology classes $[\omega_i]$ should satisfy the relation 
\begin{equation}\label{integral condition}
\int \omega_i \wedge \omega_j = \delta_{ij}.
\end{equation} 
Given a definite triple, we can take constant linear combinations of the $\omega_i$ to give a new definite triple for which (\ref{integral condition}) holds. From now on we assume that (\ref{integral condition}) is satisfied.

\begin{definition}
Let $\omega =(\omega_1, \omega_2, \omega_3)$ be a definite triple. We define a Riemannian metric $g_\omega$ on $X$ by setting $\Lambda^+_\omega$ to be the span of the $\omega_i$ and the volume form to be $\mu_\omega = \frac{1}{3}\sum \omega_i^2$.
\end{definition}

\begin{definition}
Given a definite triple $\omega = (\omega_1, \omega_2, \omega_3)$ we define a symmetric-matrix valued function $Q(\omega) \colon X \to S^2\R^3$ by 
\[
Q_{ij}(\omega) 
= 
\frac{\omega_i \wedge \omega_j}{\mu(\omega)}.
\]
\end{definition}

\begin{definition}
Given a definite triple $\omega = (\omega_1, \omega_2, \omega_3)$ we define the \emph{energy} of $\omega$ by
\[
\F(\omega) = \int_X \tr(Q(\omega)^2)\, \mu(\omega)
\]
\end{definition}

\begin{definition}
We fix a reference definite triple $\omega = (\omega_1, \omega_2, \omega_3)$. Let $\T \subset \Omega^1(X, \R^3)$ denote those triples of 1-forms $a = (a_1,a_2,a_3)$ such that $\omega + \diff a$ is again a definite triple. Given $a \in \T$, we write $\omega_a$, $Q_a$ and $\F(a)$ for the corresponding definite triple, symmetric matrix and energy.

$\T$ is naturally a Riemannian manifold, the inner-product on $T_a \T = \Omega^1(X, \R^3)$ is given by the $L^2$-inner-product. For $b,c\in T_a \T$, 
\[
\langle b,c \rangle
=
\int_X \left((b_1,c_1)+ (b_2,c_2) + (b_3, c_3)\right) \, \mu
\]
where the Riemannian inner-product $(\cdot, \cdot)$ and volume-form $\mu$ are those of the Riemannian metric associated to $\omega + \diff a$.
\end{definition}

We now proceed exactly as for definite connections, with $\T$ playing the role of $\D$ and $\F$ that of $\E$. The proofs of Propositions \ref{lower bound for triples} and \ref{flow for triples} below are nearly identical to the case of definite connections. We replace $A$ by $a$, $F_A$ by $\omega_a$ and so on. In place of each occurrence of the Bianchi identity $\diff_A F_A = 0$ we use the equation $\diff \omega_i = 0$. No other changes are necessary; consequently we do not rewrite the proofs out here.

\begin{proposition}\label{lower bound for triples}
Let $a \in \T$. There is a lower bound $\F(a) \geq 3$ with equality if and only if $\omega_a$ is a hyperkähler triple.
\end{proposition}

We also describe the gradient flow of $\F \colon \T \to \R$. Given $a \in \T$, let $B_a$ denote the symmetric-matrix valued function $B_a \colon X \to S^2 \R^3$  
\[
B_a = Q_a - \frac{1}{6}\tr(Q_a)
\]
We write $B_a(\omega_a)$ for the triple of 2-forms given by applying $B_a$ to $\omega_a$ at each point of $X$. Explicitly, if $B_a$ is the matrix with elements $B_{ij}$ then $B_a(\omega_a)$ is the triple of 2-forms whose $i^\text{th}$ element is
\[
\sum_j B_{ij} (\omega_j + \diff a_j).
\]

\begin{proposition}\label{flow for triples}
The downward gradient flow of $\F$ is given by
\begin{equation}\label{gradient flow for triples}
\frac{\del a}{\del t}
=
- \diff^* \left(B_a (\omega_a)\right).
\end{equation}
Here $\diff^*$ is the $L^2$-adjoint defined by the metric $\omega_a$. 
\end{proposition}

Explicitly, this result says that
\[
\frac{\del a_i}{\del t}
=
- \sum_j \diff^*\left( B_{ij}(\omega_j + \diff a_j)\right).
\]

To prove short-time existence of the flow (\ref{gradient flow for triples}) we again follow the argument given above for definite connections. In the case of triples, the ``gauge group'' is made up of two separate groups. Firstly, there is gauge inherent in our use of 1-forms to parametrise definite triples; secondly, the identity component  $\Diff_0(X)$ of the diffeomorphism group acts on definite triples, preserving their cohomology classes.

Infinitesimally on $\T$, this corresponds to two actions. Firstly, there is the linear action of the space $\Omega^0(X, \R^3)$ of triples of functions on $\T$, given by
\begin{equation}\label{action of functions on triples}
(f_1, f_2, f_3) \cdot (a_1, a_2, a_3) 
=
(a_1 + \diff f_1, a_2+ \diff f_2, a_3 + \diff f_3).
\end{equation}
Secondly, there is an action of vector fields; given $a\in \T$, define 
\[
S_a \colon \Omega^0(X, TX) \to \Omega^1(X, \R^3)
\] 
by $v \mapsto \iota_v \omega_a$. I.e., $S_a(v)$ is the triple of 1-forms
\begin{equation}\label{action of vectors on triples}
S_a(v)
=
\left(\iota_v \omega_1 + \iota_v(\diff a_1),
\iota_v \omega_2 + \iota_v(\diff a_2),
\iota_v \omega_3 + \iota_v(\diff a_3)\right)
\end{equation}
Note that $\diff(\iota_ v \omega_a) = L_v \omega_a$ and so this action covers the natural infinitesimal action of vector fields on 2-forms. (The action $v \mapsto L_v(a)$ does \emph{not} cover the action of vector fields on 1-forms but this is simply because we fixed a reference definite triple in the definition of $\T$, breaking the diffeomorphism invariance.)

We now proceed exactly as before. The infinitesimal actions (\ref{action of functions on triples}) and (\ref{action of vectors on triples}) prevent the flow (\ref{gradient flow for triples}) from being parabolic, but this is the only way in which parabolicity fails. Accordingly we can transform to a gauge equivalent flow which is parabolic, apply standard short-time existence results there and then transform back to prove short-time existence for (\ref{gradient flow for triples}). The proof is again identical to that given above for definite triples and so we state the outcome without giving the details.

\begin{theorem}\label{short time existence for triples}
Given $a \in \T$, there exists $\epsilon >0$ and a path $a(t) \in \T$ for $t \in [0, \epsilon)$ solving the downward gradient flow equation (\ref{gradient flow for triples}) with $a(0)= a$. The flow is unique for as long as it exists.
\end{theorem}

Finally, we come to a result whose analogue for definite connections we do \emph{not} know.

\begin{proposition}\label{unique critical point}
The only possible critical points of $\F \colon \T \to \R$ correspond to hyperkähler triples. 
\end{proposition}
\begin{proof}
Given a definite triple $ \omega_a$, it follows that the symplectic structure $\omega_1 + \diff a_1$ has vanishing first Chern class. Indeed, given a compatible almost complex structure, the almost canonical bundle is isomorphic to the span of $\omega_2 + i \omega_3$ and so is trivial.  A result of Bauer \cite{Bauer2008Almost-complex-} says that for such a symplectic 4-manifold, $b_+(X) = 3$. So the only harmonic self-dual 2-forms with respect to $g_{\omega_a}$ are constant linear combinations of the $\omega_j + \diff a_j$. 

Meanwhile, a critical point $a$ of $\F$ satisfies $\diff^* (B_a(\omega_a)) = 0$. The triple of 2-forms $B_a(\omega_a)$ is a linear combination of the components of $\omega_a$,  which are $g_{\omega_a}$ self-dual. It follows that $B_a(\omega_a)$ is a triple of self-dual and co-closed, hence harmonic 2-forms. So $a$ is a critical point of $\F$ if and only if $B_a(\omega_a)$ is a triple of harmonic self-dual 2 forms.

It follows that $B_a = Q_a - \frac{1}{6} \tr(Q_a^2)$ is a constant matrix which implies that $Q_a$ is constant. By the normalisation condition (\ref{integral condition}), $Q_a = \text{id}$ and hence $\omega_a$ is a hyperkähler triple.
\end{proof}

The proof of this result hinges on the result of Bauer and, ultimately, on the deep work of Taubes concerning the Seiberg--Witten invariants of symplectic 4-manifolds. It is interesting to ask if there is a more direct proof of Proposition \ref{unique critical point}, avoiding Seiberg--Witten theory, which could perhaps also apply in the case of definite connections.

\section{A moment-map interpretation}\label{moment map interpretation}

In this section we explain how the condition that a definite connection be perfect can be seen as the vanishing of a moment map. It will be interesting to see if the moment-map perspective can provide insight into this problem, much as it has done for other geometric PDEs (such as Hermitian--Einstein connections over K\"ahler manifolds, K\"ahler--Einstein metrics or more generally extremal K\"ahler metrics).

\subsection{A moment map for integral symplectic manifolds}

The starting point is a moment map arising in a more general situation, considered in \cite{Fine2011The-Hamiltonian}.  We describe this briefly here and refer to the original article for details. Let $L \to M$ be a Hermitian line bundle over a compact $2n$-dimensional manifold; moreover, suppose that $c_1(L)$ admits symplectic representatives. We denote by $\s$ the set of all unitary connections $A$ in $L$ whose curvature satisfies the condition that $\frac{i}{2\pi}F_A$ is a symplectic form on $M$. As we will explain $\s$ is a symplectic manifold and there is a moment-map for the action of the group $\G_L$ of bundle isometries.

The set $\s$ is open in the space of all connections (for, say, the $C^\infty$ topology). The tangent space $T_A \s$ is the space $\Omega^1(M, i\R)$ of imaginary 1-forms. In order to avoid factors of $i$ in all our formulae, we divide by $i$ at the outset and identify $T_A \s \cong \Omega^1(M, \R)$. Given $A \in \s$, we write $\omega_A$ for the associated symplectic form. Our conventions mean that $\omega_{A+a} = \omega_A + \diff a$ for $a\in \Omega^1(M,\R)$.

\begin{definition}
We define a 2-form $\Omega$ on $\s$ by
\[
\Omega_A(a,b) = \frac{1}{(n-1)!}\int_X a \wedge b \wedge \omega_A^{n-1},
\]
for $a,b \in \Omega^1(M ,\R)$.
\end{definition}

\begin{proposition}
The 2-form $\Omega$ is a symplectic form.
\end{proposition}

We remark that this picture is obviously motivated by that of Atiyah and Bott \cite{Atiyah1983The-Yang-Mills-}. They consider unitary connections in bundles of \emph{arbitrary} rank, but over a base with a \emph{fixed} symplectic form.

The group $\G_L$ of bundle isometries of $L$ (not necessarily covering the identity) acts by pull-back on $\s$, leaving $\Omega$ invariant. To describe the moment map for this action, we first note that given a connection $A$ in $L$ and $\eta \in \Lie (\G_L)$, one can define a function $A(\eta) \in C^\infty(M, \R)$. Thinking of $\eta$ as a vector field on $L$, the connection $A$ splits $\eta$ into a vertical and a horizontal part. On each fibre, the vertical part is multiplication by $iA(\eta)$.

Alternatively, we can think of a connection $A$ as an $S^1$-invariant 1-form on the principal circle bundle $P \to M$ corresponding to $L \to M$. Then $\eta$ is an $S^1$-invariant vector field on $P$ and the function $A(\eta)$ given by pairing the 1-form $A$ with the vector field $\eta$ is the function we seek, pulled back to ~$P$. (Again, normally one considers connections on principal circle bundles as \emph{imaginary} valued 1-forms, but we divide by $i$ throughout and use instead real 1-forms.) 

\begin{proposition}
The map $m \colon \s \to (\Lie(\G_L))^*$ defined by
\[
\langle m (A), \eta \rangle
=
\frac{1}{n!}\int_M A(\eta)\, \omega_A^n
\]
is a $\G_L$-equivariant moment map for the action of $\G_L$ on $\s$.
\end{proposition}

\subsection{Definite connections as an isotropic subspace}

We now return to our discussion of definite connections. Our goal in this section is to realise the space of definite connections as an isotropic subspace of an infinite dimensional symplectic manifold $\s$ of the kind just described. 

We first recall the symplectic interpretation of definite connections explained in \cite{Fine2009Symplectic-Cala}. We consider an $\SO(3)$-bundle $E \to X$ over a 4-manifold and the unit sphere bundle $\pi \colon Z \to X$. Write $V \to Z$ for the vertical tangent bundle, an $\SO(2)$-vector bundle. 

\begin{definition}\label{induced unitary connection}
Given a metric connection $A$ in $E$, we define a metric connection $A_V$ in the bundle $V$ as follows. Write $TZ = V \oplus H$, where $H$ is the horizontal complement provided by $A$. In the vertical directions, $A_V$ is the Levi-Civita connection of the fibres. To define $A_V$ horizontally, let $u \in T_pZ$ be an $A$-horizontal tangent vector and let $\gamma \colon (-\epsilon, \epsilon) \to Z$ be a horizontal path with $\gamma'(0) = u$. Parallel transport with respect to $A$ along the path $\pi \circ \gamma$ in $X$ trivialises $(\pi \circ \gamma)^*Z$ and hence $\gamma^* V$ over $(-\epsilon, \epsilon)$. Given a section $s$ of $V \to Z$, we define the $A_V$-covariant derivative of $s$ in the direction $u$ by $\frac{\diff}{\diff t} (\gamma^*s)$ in this trivialisation.

A more highbrow way to define $A_V$ is to think of $A$ as a connection on the principal frame bundle of $E \to X$ which then induces connections on all associated bundles. Let $\mathcal V \to X$ denote the vector bundle whose fibre at $x \in X$ is $C^\infty(TZ_x)$, the space of all vector fields on the $S^2$-fibre of $Z$ over $x$. $\mathcal V$ is associated to $E$ via the action of $\SO(3)$ on $C^\infty(TS^2)$; so $A$ induces a connection in $\mathcal V$, which we also write as $A$. Of course, a section of $\mathcal V \to X$ is the same thing as a section of $V \to Z$. Now let $s$ be a section of $V\to Z$ and $w$ a tangent vector field on $X$ with $A$-horizontal lift $u$ on $Z$. We define 
\[
\nabla^{A_V}_u(s)
=
\nabla^A_w(s)
\]
where on the right-hand-side we interpret $s$ as a section of $\mathcal V \to X$.
We note for later use that if $a \in \Omega^1(X, \so(E))$, then 
\[
\nabla^{A+a}_w(s) = \nabla^A_w(s) + [a(w),s]
\]
where on the right-hand-side we interpret $a(w) \in \so(E)$ as a section of $\mathcal V$ and the Lie bracket is that of vector fields on the sphere. 
\end{definition}

Choosing a fibrewise orientation in $Z \to X$ makes $V$ an oriented $\SO(2)$-bundle and so we can think of it as a Hermitian line bundle (multiplication by $i$ is a positive rotation by $\pi$) and $A_V$ as a unitary connection. The key point for us it that the curvature of $A_V$ is a symplectic form on $Z$ if and only if $A$ is a definite connection. (This is proved in~\cite{Fine2009Symplectic-Cala}.)

So we can view definite connections using the set-up described above. $V \to Z$ is the integral symplectic manifold and the map $A \mapsto A_V$ gives an embedding $\D \to \s$. Given a definite connection $A$, we write $\omega_A = \omega_{A_V}$ for the corresponding symplectic form on~$Z$. 

We now consider the group $\G$ of bundle maps of $E$ which preserves both the fibrewise orientation and metrics. $\G$ acts by fibrewise orientation-preserving isometries on the sphere bundle $Z$ and hence by orientation-preserving bundle isometries on the vertical tangent bundle $V$. This gives a natural inclusion $\G \to \G_V$ into the group of unitary bundle maps of $V$, thought of as a Hermitian line bundle. With this understood, the map $\D \to \s$ is $\G$-equivariant.

To describe this embedding in more detail, we first define a linear map 
\[
h \colon \Omega^0(X, \so(E)) \to \Omega^0(Z, \R)
\]
as follows. To a section $\rho$ of $\so(E)$ we associate the function $h(\rho) \colon Z \to \R$ whose restriction to the fibre over $x \in X$ is the mean-value zero Hamiltonian of the rotation $\rho(x)$. (We use the area form of the round metric on each fibre of $Z$ to define the Hamiltonians here.) Now, tensoring $h$ with pull-back $\pi^* \colon \Omega^p(X,\R) \to \Omega^p(Z, \R)$ gives maps on $\so(E)$-valued $p$-forms, $\Omega^p(X, \so(E)) \to \Omega^p(Z, \R)$ which we also denote by $h$. 

For the next result, we write $\A_E$ for the space of all metric connections in $E$, an affine space modelled on $\Omega^1(X, \so(E))$. Similarly, we write $\A_V$ for the space of all unitary connections in $V$,  an affine space modelled on $\Omega^1(M, \R)$. 

\begin{lemma}\label{embedding is h}
The embedding $\A_E \to \A_V$ given by Definition \ref{induced unitary connection} is affine, modelled on the map 
\[
-2 \pi h \colon \Omega^1(X, \so(E)) \to \Omega^1(Z, \R).
\]
\end{lemma}
\begin{proof}
Let $A, \hat A \in \A_E$, with $\hat A = A + a$ for $a\in \Omega^1(X, \so(E))$. We must show that 
\[
\hat A_V - A_V = -2\pi h(\hat A - A).
\]
Let $s \in C^\infty(Z,V)$. By definition 
\[
\nabla^{\hat A_V}_{v}(s)
=
\nabla^{A_V}_{v}(s)
\]
whenever $v$ is vertical. So the 1-form $\hat A_V - A_V$ vanishes on vertical vectors, as does $h(\hat A - A)$.

Let $w \in C^\infty(X,TX)$ and write $u$ for its $A$-horizontal lift to a vector field on $Z$. Then its $\hat A$-horizontal lift to $Z$ is $\hat{u} = u + a(w)$, where here and throughout the proof we identify $a(w) \in \Omega^0(X, \so(E))$ with the corresponding vertical vector field on $Z$. So, given a section $s \in C^\infty(Z, V)$, we have
\[
\nabla_u^{\hat A_V} (s)
=
\nabla_{\hat u}^{\hat A_V}(s)
-\nabla_{a(w)}^{S^2}(s)
\]
where we have written $\nabla^{S^2}$ for the Levi-Civita connection on $S^2$.

Now, $\nabla_{\hat u}^{\hat A_V}$ and $\nabla_u^{A_V}$ are covariant derivatives corresponding to the same vector field $w$ downstairs on $X$. It follows that:
\[
\nabla^{\hat A_V}_{\hat u}(s)
=
\nabla^{A_V}_u(s) + [a(w),s].
\]
(See the discussion at the end of Definition \ref{induced unitary connection}.) But the Levi-Civita connection on $S^2$ is torsion free, so assembling the pieces gives
\[
\nabla_u^{\hat A_V} (s)
=
\nabla^{A_V}_u(s)
-\nabla_s^{S^2}(a(w)).
\] 

This reduces our calculation to the following question purely about the geometry of $S^2$: given $\rho \in \so(3)$, the map $T_qS^2 \to T_qS^2$ given by $s \mapsto \nabla^{S^2}_s \rho$ is a rotation by what angle? The following lemma ensures that this is $2\pi h(\rho)(q)$ where $h(\rho)$ is the Hamiltonian of $\rho$. Applying this with $\rho = a(w)$ gives the claimed formula for $\hat A_V - A_V$.
\end{proof}

\begin{lemma}
Let $\rho \in \so(3)$. The map $T_qS^2 \to T_qS^2$ given by $s \mapsto \nabla^{S^2}_s \rho$ is a rotation by $2\pi h(\rho)(q)$.
\end{lemma}

\begin{proof}
Think of $\rho$ as a vector field on $\R^3$; it is given simply by matrix multiplication: $x\mapsto \rho(x)$.  Since $\rho(x)$ is linear in $x$, $\nabla^{\R^3}_s\rho = \rho(s)$. The connection on $S^2$ is induced by projection, giving $\nabla^{S^2}_s \rho =  \rho(s) - (\rho(s), q)\, q$. In other words, $s \mapsto \nabla_s^{S^2}\rho$ is given by the component of $\rho$ which is rotation about the axis through $q$. The size of this component is exactly $2\pi h(\rho)(q)$.
\end{proof}

We will also need the following result, proved in \cite{Fine2009Symplectic-Cala} via a similar calculation as in the previous lemma.

\begin{proposition}\label{decomposition of omega}
Let $A \in \A_E$ and write the resulting vertical-horizontal decomposition of 2-forms on $Z$ as
\[
\Lambda^2(T^*Z) \cong \Lambda^2 V^* \oplus (V^* \otimes H^*) \oplus \Lambda^2H^*
\]
With respect to this decomposition, the curvature of the induced connection $A_V \in \A_V$ is $F_{A_V} = -2\pi i \omega_{A_V}$ where
\[
\omega_{A_V} = \left(\omega_{S^2}\oplus 0\oplus  h(F_A)\right).
\]
(Here $\omega_{S^2}$ denotes the fibrewise area form.)
\end{proposition}

We are now in a position to prove:

\begin{proposition}
The image of the embedding $\D \to \s$ is isotropic.
\end{proposition}
\begin{proof}
Applying the antipodal map $S^2 \to S^2$ fibrewise on $Z$ gives an involution $\gamma \colon Z \to Z$. Since the antipodal map on $S^2$ is an orientation reversing isometry, mean-value zero Hamiltonians generating rotations change sign under pull-back. Hence $\gamma^* \circ h = - h$ on $\Omega^0(X,\so(E))$. It follows that $\gamma^* \circ h = -h$ on $\Omega^p(X, \so(E))$ since $\gamma^* \circ \pi^* = \pi^*$.

From Proposition \ref{decomposition of omega} we see that, given a definite connection $A$,  $\gamma^* \omega_{A_V} = - \omega_{A_V}$. In particular, $\gamma$ is orientation reversing. Now, let $a, b \in T_A \D = \Omega^1(X, \so(E))$. It follows from Lemma \ref{embedding is h} that under the embedding $\D \to \s$, they correspond to tangent vectors $h(a),h(b) \in T_{A_V}\s = \Omega^1(Z, \R)$. Evaluating the symplectic form on $\s$ on them gives
\[
\Omega(h(a),h(b))
=
\int_Z h(a) \wedge h(b) \wedge \omega_{A_V}^2
\]
Since $\gamma^*\circ h = -h$ and $\gamma^*\omega_{A_V} = -\omega_{A_V}$, the integrand is $\gamma$-invariant; on the other hand, $\gamma$ also reverses orientation. Hence the integral vanishes.
\end{proof}

\subsection{A moment map for definite connections}

In this section we explain how perfect connections---and hence anti-self-dual Einstein metrics of non-zero scalar curvature---are the zeros of a moment map. To begin, we reformulate the condition of being perfect in the current notation.

\begin{lemma}
A definite connection $A$ is perfect if and only if $h(F_A)^2 \in \Omega^4(Z, \R)$ is pulled back from $X$.
\end{lemma}
\begin{proof}
Let $p \in Z$, regarded as a unit length vector in a fibre of $E$. Then $h(F_A)^2(p) = Q(A)(p,p) \mu(A)$.
\end{proof}

To display this as the vanishing of a moment map we consider a certain subalgebra of the Lie algebra $\Lie(\G_V)$ of infinitesimal bundle isometries of $V$. The sub-algebra will be the Lie algebra of the group $\G_\pi$  of diffeomorphisms of $Z$ which leave each fibre invariant and, moreover, preserve the area forms on each fibre. (In particular, they cover the identity downstairs on $X$.) We will lift the action of $\Lie(\G_\pi)$ on $Z$ to $V \to Z$. This does not integrate up to an action of the whole group, but this is not important for the definition of a moment map.

To do this we first make a short digression to recall a standard fact. Let $L \to (M, \omega)$ be a compact integral symplectic manifold with $A$ a unitary connection in $L$ with curvature $\omega$; given a Hamiltonian vector field $v$ on $M$ with mean-value zero Hamiltonian $f$, denote by 
\[
\hat{v} = \tilde v + f(v) \frac{\del }{\del \theta}
\]
the vector field on $L$ where $\tilde v$ is the $A$-horizontal lift of $v$ and $\del/\del \theta$ generates the fibrewise $S^1$-action. The fact we need is the following.

\begin{lemma}\label{lifting Hamiltonians}
The map $v \mapsto \hat v$ is a Lie algebra homomorphism 
\[
\HVect(\omega_A) \to \Lie(\G_L)
\]
whose image is contained in the Lie algebra of the stabiliser of $A$. 
\end{lemma}

Note that it is not true in general that this homomorphism integrates up to a homomorphism of groups $\Ham(\omega_A) \to \G_L$. For example, considering $\O(1) \to \C\P^1$, the subgroup $\SO(3) \subset \Ham(\C\P^1)$ gives rise to a subalgebra of $\Lie(\G_L)$ which integrates up to a copy of $\SU(2)$ rather than $\SO(3)$. 

With this fact in hand we can now explain how to lift the action of $\Lie(\G_\pi)$ on $Z$ to $V \to Z$. Note that, since $\pi_1(S^2) = 1$, area preserving diffeomorphisms of $S^2$ are Hamiltonian. Next we lift elements of $\HVect(S^2)$ to $TS^2$, via Lemma \ref{lifting Hamiltonians} using the Levi-Civita connection on $TS^2$. Doing this on every fibre of $V \to Z \to X$, we see that the action of $\Lie(G_\pi)$ lifts to an action on $V$ by infinitesimal bundle isometries.

The resulting action of $\Lie(\G_\pi)$ on $\s$ has a moment map given by projecting the original moment map $m \colon \s \to (\Lie \G_V)^*$ to $(\Lie \G_\pi)^*$. We denote this by 
\[
m_\pi \colon \s \to (\Lie \G_\pi)^*.
\]  
We now show that $m_\pi$ is the moment map we are looking for.

\begin{theorem}
Let $A\in \D$ be a definite connection. Then $m_\pi(A_V) = 0$ if and only if $A$ is perfect.
\end{theorem}

\begin{proof}
The condition $m_\pi(A_V) = 0$ says that $\int A_V(\eta) \omega_{A_V}^3 = 0$ for all $\eta \in \Lie \G_\pi$. We think of $\eta \in \Lie \G_\pi$ as a vector field on the principal $S^1$-bundle $P\to Z$. By definition, $\eta$ is tangent to the fibres of $P \to Z \to X$. Moreover, using the Levi-Civita connection on each fibre we can split $\eta$ into two components:
$$
\eta = A(\eta)\frac{\del}{\del \theta} + v.
$$
It follows from the definition of $\Lie(\G_\pi) \to \Lie(\G_V)$ that $A(\eta) = h(v)$ and, in particular, has fibrewise mean-value zero. Conversely, every function $Z \to \R$ which has fibrewise mean-value zero arises as $A(\eta)$ for some $\eta$. So, $A \in \D$ has $m_\pi(A_V) = 0$ if and only if $\int_Z f \omega_{A_V}^3=0$ for every function $f \colon Z\to \R$ with fibrewise mean-value zero. 

Note that the volume form is
$$
\omega_{A_V}^3 = \omega_{V} \wedge h(F_A)^2.
$$
First suppose that $A$ is perfect so that $h(F_A)^2 = \pi^*\alpha$ is pulled back from $X$. Then, for any function $f \colon Z \to \R$, 
$$
\int_Z f \omega_{A_V}^3 
= 
\int_X \left( \pi_*( f \omega_{V}) \right)\alpha .
$$
In particular, for any function with fibrewise mean-value zero, the function $\pi_*(f \omega_{V})$ vanishes and hence $m_\pi(A_V)=0$. 

Conversely, suppose $m_\pi(A_V) = 0$. We will show that $h(F_A)^2$ is pulled back from the base. Let $\alpha = \pi_*(\omega_{A_V}^3)$. Since $\pi_*(\omega_{A_V}^3)$ is a volume form on $X$, $\pi^*\alpha$ spans $\Lambda^4H^*$ at every point of $Z$. Hence we can write
$$
h(F_A)^2 = (1 + \chi) \pi^*\alpha
$$
for some function $\chi$ which has fibrewise mean-value zero. Since $m_\pi(A_V) = 0$, we have that
$$
0 
=
\int_Z \chi\, \omega_{A_V}^3 
=
\int_Z (1+\chi)\chi \,\omega_{V}\wedge \pi^*\alpha
= 
\int_Z \chi^2\, \omega_{V}\wedge \pi^*\alpha 
$$
But $\omega_{V}\wedge \pi^*\alpha$ is a volume form on $Z$, so this forces $\chi = 0$ and $h(F_A)^2$ is pulled back from the base as required.
\end{proof}
{\small
\bibliographystyle{plain}
\bibliography{asdE_gauge_bibliography}
}
\end{document}